\documentclass[final,onefignum,onetabnum]{siamart171218}
\usepackage{lipsum}
\usepackage{amsfonts}
\usepackage{graphicx}
\usepackage{epstopdf}
\usepackage{algorithm,algorithmicx,algpseudocode}
\usepackage{multirow}
\usepackage{booktabs}
\usepackage{subfigure}
\usepackage{scalerel}
\ifpdf
  \DeclareGraphicsExtensions{.eps,.pdf,.png,.jpg}
\else
  \DeclareGraphicsExtensions{.eps}
\fi

\newsiamremark{remark}{Remark}
\newsiamremark{hypothesis}{Hypothesis}
\crefname{hypothesis}{Hypothesis}{Hypotheses}
\newsiamthm{claim}{Claim}

\headers{Global randomized block Kaczmarz algorithm}{Yu-Qi Niu and Bing Zheng}

\title{On global randomized block Kaczmarz algorithm for solving large-scale matrix equations\thanks{Submitted to the editors DATE.
\funding{This work was supported by the National Natural Science Foundation of China (Grant No. 12071196).}}}

\author{Yu-Qi Niu
\and Bing Zheng\thanks{Corresponding author. School of Mathematics and Statistics, Lanzhou University, Lanzhou 730000, People’s Republic of China (\email{bzheng@lzu.edu.cn}).}
}

\usepackage{amsopn}
\DeclareMathOperator{\diag}{diag}

\ifpdf
\hypersetup{
  pdftitle={Global randomized block Kaczmar algorithm},
  pdfauthor={Yu-Qi Niu and Bing Zheng}
}
\fi

\begin{document}

\maketitle

% REQUIRED
\begin{abstract}
The randomized Kaczmarz algorithm is one of the most popular approaches for solving large-scale linear systems due to its simplicity and efficiency. In this paper, we propose two classes of global randomized Kaczmarz methods for solving large-scale linear matrix equations $AXB=C$, the global randomized block Kaczmarz algorithm and global randomized average block Kaczmarz algorithm. The feature of global randomized block Kaczmarz algorithm is the fact that the current iterate is projected onto the solution space of the sketched matrix equation at each iteration, while the global randomized average block Kaczmarz approach is pseudoinverse-free and therefore can be deployed on parallel computing units to achieve significant improvements in the computational time. We prove that these two methods linearly converge in the mean square to the minimum norm solution $X_*=A^\dag CB^\dag$ of a given linear matrix equation. The convergence rates depend on the geometric properties of the data matrices and their submatrices and on the size of the blocks. Numerical results reveal that our proposed algorithms are efficient and effective for solving large-scale matrix equations. In particular, they can also achieve satisfying performance when applied to image deblurring problems.

%The randomized Kaczmarz method is one of the most popular algorithms for solving large-scale linear systems. In this paper, we first propose a global randomized block Kaczmarz algorithm for solving large-scale matrix equations. At each iteration, the global randomized block Kaczmarz algorithm projects the current iterate onto the solution space of the sketched matrix equation. Furthermore, we present a class of global randomized average block Kaczmarz algorithms for solving large-scale matrix equations, which uses two randomized index sets and constant or adaptive stepsize at each iteration. The class of algorithms need not compute the pseudoinverses, which is different from the global randomized block Kaczmarz algorithm. We prove that the global randomized average block Kaczmarz algorithms converge linearly in expectation, and the convergence rates depend on the geometric properties of the matrices and their submatrices. Moreover, the global randomized average block Kaczmarz algorithms can be deployed on parallel computing units to reduce the computational time. Finally, numerical results reveal that our proposed algorithms are effective for solving large-scale matrix equations. The matrix equation $AXB=C$ is a mathematical model for deblurring problems, and our proposed methods can better reconstruct blurred images corresponding to this matrix equation.
\end{abstract}

% REQUIRED
\begin{keywords}
Matrix equations, Randomized block Kaczmarz, Average block Kaczmarz, Expected linear convergence, Image deblurring problem
\end{keywords}

% REQUIRED
\begin{AMS}
15A24, 65F45, 65F20
\end{AMS}

\section{Introduction}
Consider solving the large-scale matrix equation
\begin{equation}\label{eq:matrix equations}
AXB= C,
\end{equation}
where $A \in \mathbb{R}^{m \times p}$, $B \in \mathbb{R}^{q\times n}$, $C \in  \mathbb{R}^{m\times n}$, and unknown matrix $X \in \mathbb{R}^{p\times q}$. The large-scale linear matrix equation \eqref{eq:matrix equations} arises in many applications such as signal processing \cite{Regalia1989}, photogrammetry \cite{Rauhala1980}, etc. We assume that the equation \eqref{eq:matrix equations} has a solution, i.e., $A^\dag ACBB^\dag=C$. In practice, it is usually enough to find an approximate solution not far away from the minimal Frobenius norm solution $X_*=A^\dag CB^\dag$. Hence how to effectively solve this equation has been important research recently.

Currently, there are many direct methods based on matrix factorization \cite{Chu1987,Fausett1994,Zha1995} and iteration methods \cite{Tian2017,Peng2005,Ding2005,Ding2008,Peng2010-1} for solving large-scale linear matrix equations. However, these methods were designed with ignoring big data matrix equation problems. When applied to matrix equations from signal processing, machine learning, and image restoration, these methods can be infeasible due to exceeding storage location or requiring more time cost. In particular, the matrix equation \eqref{eq:matrix equations} can be written the following equivalent matrix-vector form by the Kronecker product
\begin{equation}\label{eq:systems}
 (B^T\otimes A){\rm vec}(X)={\rm vec}(C),
\end{equation}
where the Kronecker product $(B^T\otimes A)\in \mathbb{R}^{mn \times pq}$, the right-side vector ${\rm vec}(C)\in \mathbb{R}^{mn \times1}$, and the unknown vector ${\rm vec}(X)\in \mathbb{R}^{pq \times1}$. Many iteration methods are proposed to solve the matrix equation \eqref{eq:matrix equations} by applying the Kronecker product, see \cite{Zhang2011,Cvetk2008,Peng2010-2}. Recently, randomized methods have been widely concerned for solving large-scale linear systems, such as randomized SVD \cite{Wei2016,Wei2019,Wei2020},  the randomized Kaczmarz algorithm \cite{Strohmer2009,Zouzias2012,Needell2015,Ma2015,Du20}, and the randomized coordinate descent method \cite{Leventhal2008,Gower,Du2019,Du21}. Randomized Kaczmarz algorithm applied to the linear system \eqref{eq:systems} can also solve the matrix equation \eqref{eq:matrix equations}. Nevertheless, when the dimensions of matrices $A$ and $B$ are large, the dimensions of the linear system \eqref{eq:systems} increase dramatically, which causes these randomized projection algorithms to require extra cache memory and too much time. Du et al. proposed the randomized block coordinate descent (RBCD) method for solving the matrix least-squares problem $\min_{X\in \mathbb{R}^{p\times q}}\|C-AXB\|_F$ in \cite{Du22}. This method is computationally expensive per iteration because it involves large-scale matrix-matrix products.

In this paper, we propose the global block randomized block Kaczmarz (GRBK) algorithm for solving the large-scale matrix equations by determining the solution of the matrix equation \eqref{eq:matrix equations} from the randomized sketched matrix equation \eqref{eq:sketched}. In practice, to avoid computing pseudoinverse, we study a parallelized version of GRBK, in which a weighted average of independent updates is used. Before summarizing our contributions, we first present the notations that are used throughout this paper and briefly describe the properties of the Kronecker product.

\subsection{Notation}
For an integer $m\geq1$, let $[m]=\{1,2,\cdots,m\}$. We denote by $\mathbb{R}^{m\times n}$ the space of all $m\times n$ real matrices, and by $\|\cdot\|_2$ and $\|\cdot\|_F$ the 2-norm and Frobenius norm, respectively. Given two matrices $X,Y\in \mathbb{R}^{n\times n}$, $\langle X,Y\rangle_F={\rm tr}(X^TY)$ is the Frobenius inner product of matrices $X$ and $Y$ and ${\rm tr}(X)$ denotes the trace of $X$, specially, $\langle X,X\rangle_F=\|X\|_F^2$. In addition, for any matrix $A\in^{m\times n}$, we use $A^T$, $A^\dag$, $A_{i,:}$, $A_{:,j}$ to denote the transpose, the pseudoinverse, the $i$th row, the $j$th column of $A$, respectively. We also denote the maximum singular value and minimum nonzero singular value of $X$ by $\sigma_{\max}(X)$ and $\sigma_{\min}(X)$, respectively. For index sets $I\subset[m]$ and $J\subset[n]$, let $A_{I,:}$, $A_{:,J}$ and $A_{I,J}$ denote the row submatrix indexed by $I$, the column submatrix indexed by $J$, and the submatrix that lies in the rows indexed by $I$ and the columns indexed by $J$, respectively. We use $|I|$ to denote the cardinality of a subset $I\subset[m]$. For any random variable $X$, let $\mathbb{E}[X]$ denote its expectation.

\subsection{Kronecker product}
For deriving the convergence of the algorithms in this paper, the Kronecker product is used. We briefly state a few of its useful properties here. More can be found, e.g. in \cite{Graham1981}. For all matrices $A$ and $B$, we have
\begin{equation*}
(A\otimes B)^\dag=A^\dag\otimes B^\dag,~ (A\otimes B)^T=A^T\otimes B^T, ~\|A\otimes B\|_F=\|A\|_F\|B\|_F,
\end{equation*}
and
\begin{equation*}
\sigma_{\max}(A\otimes B)=\sigma_{\max}(A)\sigma_{\max}(B), ~\sigma_{\min}(A\otimes B)=\sigma_{\min}(A)\sigma_{\min}(B).
\end{equation*}
Furthermore, we introduce the ${\rm vec}(\cdot)$ operation by stacking columns. If $X\in \mathbb{R}^{m\times n}$, then
\begin{equation*}
{\rm vec}(X)=\begin{pmatrix}X_{:,1}^T&  \cdots& X_{:,n}^T \end{pmatrix}^T\in \mathbb{R}^{mn\times 1}.
\end{equation*}
If $A\in \mathbb{R}^{m \times p}$, $X\in \mathbb{R}^{p\times q}$, and $B\in \mathbb{R}^{q\times n}$, then
\begin{equation*}
{\rm vec}(AXB)=(B^T\otimes A){\rm vec}(X).
\end{equation*}
\subsection{Contributions}
We propose the global randomized block Kaczmarz (GRBK) algorithm to solve the matrix equations $AXB=C$. The GRBK method uses two randomized strategies to choose two subsets $I_k$ and $J_k$ of the constraints at each iteration, and updates (see \Cref{alg:GRBK}):
\begin{equation*}
X_{k+1}=X_k+A_{I_k,:}^\dag(C_{I_k,J_k}-A_{I_k,:}X_kB_{:,J_k})B_{:,J_k}^\dag.
\end{equation*}
We provide a theoretical guarantee for the GRBK method and demonstrate its performance. Furthermore, to avoid computing pseudoinverses, we propose a global randomized average block Kaczmarz method to solve the matrix equations $AXB=C$ and denote the GRABK method (see \Cref{alg:GRABK}). The updated format of the GRABK method is as follows:
\begin{equation*}
X_{k+1}=X_k+\alpha_k\left(\sum_{i\in I_k,j\in J_k}u_i^{k}v_j^{k}\tfrac{A_{i,:}^T(C_{ij}-A_{i,:}X_kB_{:,j})B_{:,j}^T}{\|A_{i,:}\|^2\|B_{:,j}\|^2}\right),
\end{equation*}
where the stepsize $\alpha_k\in(0,2)$ and the weights $u_i^{k},v_j^{k}\in [0,1]$ such that $\sum_{i\in I_k}u_i^{k}=1$ and $\sum_{j\in J_k}v_j^{k}=1$. In addition, we analyzed the convergence of the GRBK and GRABK methods.
\subsection{Outline}
In \Cref{Sec: GRBK}, we derive the global randomized block Kaczmarz method for solving the large-scale matrix equation. In \Cref{Sec: GRABK}, we define the global randomized average block Kaczmarz methods for solving the large-scale matrix equation and derive new convergence rates. In \Cref{Sec:NR}, we report the numerical results to corroborate our theoretical. Finally, we state brief conclusions in \Cref{Sec:Con}.

\section{The global randomized block Kaczmarz algorithm }\label{Sec: GRBK}
In this paper, we are concerned with the randomized Kaczmarz method to solve the matrix equation $AXB=C$. Since matrix equation \eqref{eq:matrix equations} is difficult to solve directly, our way is to iteratively solve a small randomized version of matrix equation \eqref{eq:matrix equations}. This is, we choose two index sets $I\subseteq[m]$ and  $J\subseteq[n]$ at random, and instead solve the following sketched matrix equation:
\begin{equation}\label{eq:sketched}
A_{I,:}XB_{:,J}=C_{I,J}.
\end{equation}
The sketched matrix equation is of a much smaller dimension than the original one, and hence easier to solve. However, the equation \eqref{eq:sketched} will no longer have a unique solution. In order to construct a method, we need a way of picking a particular solution. Our method defines $X_{k+1}$ to be the solution that is closest to the current iteration $X_k$ in the Frobenius norm. Hence, the next iterate $X_{k+1}$ is the nearest point to $X_k$ that satisfies a sketched matrix equation:
\begin{equation}\label{eq:it-1}
X_{k+1}=\mathop{\arg\min}_{\scaleto{\begin{smallmatrix} A_{I,:}XB_{:,J}=C_{I,J}\\X \in \mathbb{R}^{p\times q} \end{smallmatrix}}{10pt}}\tfrac{1}{2}\|X-X_k\|_F^2.
\end{equation}
In addition, $X_{k+1}$ is the best approximation of $X_*$ in a subspace passing through $X_k$:
\begin{equation}\label{eq:it-2}
\begin{split}
X_{k+1}=&\mathop{\arg\min}_{\scaleto{\begin{smallmatrix}X \in \mathbb{R}^{p\times q}\\ Y\in\mathbb{R}^{|I|\times |J|}\end{smallmatrix}}{10pt}} \tfrac{1}{2}\|X-X_*\|_F^2\\
&{\rm subject~to~~} X=X_k+A_{I,:}^TYB_{:,J}^T,~Y {\rm ~is~free}.
\end{split}
\end{equation}
By substituting the constraint \eqref{eq:it-2} into the objective function, then differentiating with respect to $Y$ to find the stationary point
\begin{equation*}
Y_*=(A_{I,:}A_{I,:}^T)^\dag(C_{I,J}-A_{I,:}X_kB_{:,J})(B_{:,J}^TB_{:,J})^\dag,
\end{equation*}
we obtain  that
\begin{eqnarray}\label{eq:it-3}
\nonumber X_{k+1}&=&X_k+A_{I,:}^T(A_{I,:}A_{I,:}^T)^\dag(C_{I,J}-A_{I,:}X_kB_{:,J})(B_{:,J}^TB_{:,J})^\dag B_{:,J}^T\\
&=&X_k+A_{I,:}^\dag(C_{I,J}-A_{I,:}X_kB_{:,J})B_{:,J}^\dag
\end{eqnarray}
is the solution to \eqref{eq:it-2}. Next, we show the equivalence of \eqref{eq:it-1} and \eqref{eq:it-2} by using Lagrangian duality. The problem \eqref{eq:it-1} has a convex quadratic objective function with linear constraints, hence strong duality holds. Introducing Lagrange multiplier $Y\in \mathbb{R}^{|I|\times |J|}$, the Lagrangian $\mathcal{L}:\mathbb{R}^{p\times q}\times \mathbb{R}^{|I|\times |J|} \rightarrow \mathbb{R}$ associated with the problem \eqref{eq:it-1} as
\begin{equation}\label{eq:dual}
\mathcal{L}(X,Y)=\tfrac{1}{2}\|X-X_k\|_F^2-\langle Y, A_{I,:}XB_{:,J}-C_{I,J}\rangle_F.
\end{equation}
Clearly, the optimal value of the primal problem \eqref{eq:it-1} is
\begin{equation*}
\min_{X \in \mathbb{R}^{p\times q}}\max_{Y\in\mathbb{R}^{|I|\times |J|}}\mathcal{L}(X,Y)
= \min_{\scaleto{\begin{smallmatrix} A_{I,:}XB_{:,J}=C_{I,J}\\X \in \mathbb{R}^{p\times q} \end{smallmatrix}}{10pt}}\tfrac{1}{2}\|X-X_k\|_F^2.
\end{equation*}
The Lagrange dual function $\mathcal{G}:\mathbb{R}^{|I|\times |J|} \rightarrow \mathbb{R}$ as the minimum value of the Lagrangian $\mathcal{L}(X,Y)$ over $X$, i.e.,
\begin{equation*}
\mathcal{G}(Y)=\min_{X \in \mathbb{R}^{p\times q}}\mathcal{L}(X,Y).
\end{equation*}
Differentiating Lagrangian $\mathcal{L}(X,Y)$ with respect to $X$ and setting to zero gives $X=X_k+A_{I,:}^TYB_{:,J}^T$. Substituting into \eqref{eq:dual} gives
\begin{eqnarray*}
\mathcal{L}(X,Y)&=&\tfrac{1}{2}\|A_{I,:}^TYB_{:,J}^T\|_F^2-\langle Y, A_{I,:}(X_k+A_{I,:}^TYB_{:,J}^T)B_{:,J}-C_{I,J}\rangle_F\\
&=& -\tfrac{1}{2}\|A_{I,:}^TYB_{:,J}^T\|_F^2-\langle Y, A_{I,:}(X_k-X_*)B_{:,J}\rangle_F\\
&=& -\tfrac{1}{2}\|A_{I,:}^TYB_{:,J}^T+X_k-X_*\|_F^2+\tfrac{1}{2}\|X_k-X_*\|_F^2.
\end{eqnarray*}
As the term $\tfrac{1}{2}\|X_k-X_*\|_F^2$ does not depend on $X$ and $Y$, substituting $X=X_k+A_{I,:}^TYB_{:,J}^T$ into the last equation, we obtain the dual problem:
\begin{eqnarray*}
\max_{Y}\mathcal{G}(Y)
&=&\max_{\scaleto{\begin{smallmatrix}X=X_k+A_{I,:}^TYB_{:,J}^T \\ X \in \mathbb{R}^{p\times q},Y\in\mathbb{R}^{|I|\times |J|} \end{smallmatrix}}{12pt}}-\tfrac{1}{2}\|X-X_*\|_F^2\\
&=&\min_{\scaleto{\begin{smallmatrix}X=X_k+A_{I,:}^TYB_{:,J}^T \\X \in \mathbb{R}^{p\times q},Y\in\mathbb{R}^{|I|\times |J|} \end{smallmatrix}}{12pt}}\tfrac{1}{2}\|X-X_*\|_F^2.
\end{eqnarray*}
Hence, by strong duality, i.e.,
\begin{equation*}
\min_{X \in \mathbb{R}^{p\times q}}\max_{Y\in\mathbb{R}^{|I|\times |J|}}\mathcal{L}(X,Y)=\max_{Y\in\mathbb{R}^{|I|\times |J|}}\min_{X \in \mathbb{R}^{p\times q}}\mathcal{L}(X,Y),
\end{equation*}
we have the equivalence of \eqref{eq:it-1} and \eqref{eq:it-2}.

Based on \eqref{eq:it-1}, \eqref{eq:it-2} and \eqref{eq:it-3}, we can summarize the methods described in this section as \Cref{alg:GRBK}, which is called the global  randomized block Kaczmarz (GRBK) method.

\begin{algorithm}[!htbp]
\caption{Global Randomized Block Kaczmarz (GRBK)}
\label{alg:GRBK}
\hspace*{0.02in}{\bf Input:} {$A\in \mathbb{R}^{m \times p}$, $B\in \mathbb{R}^{q\times n}$, and $C\in \mathbb{R}^{m \times n}$, $X_0\in \mathbb{R}^{p\times q}$.}\\
\hspace*{0.02in}{\bf Output:} {Last iterate $X_{k+1}$.}
\begin{algorithmic}[1]
\For{$k=0,1,2,\cdots,$}
\State{Select a index set $I_k\subseteq[m]$ with probability $\mathbb{P}(I_k)>0$ such that $$\sum_{I_k\subseteq[m]}\mathbb{P}(I_k)=1;$$}
\State{Select a index set $J_k\subseteq[n]$ with probability $\mathbb{P}(J_k)>0$ such that $$\sum_{J_k\subseteq[n]}\mathbb{P}(J_k)=1;$$}
\State{Update $X_{k+1}=X_k+A_{I_k,:}^\dag(C_{I_k,J_k}-A_{I_k,:}X_kB_{:,J_k})B_{:,J_k}^\dag$;}
\EndFor
\end{algorithmic}
\end{algorithm}

At each iteration, the current iterate $X_k$ is projected onto the solution space of the sketched matrix equation $A_{I_k,:}XB_{:, J_k}=C_{I_k, J_k}$.  The index sets $I_k\subseteq[m]$ and $J_k\subseteq[n]$ are selected according to probability distribution $\mathbb{P}(I_k)$ and $\mathbb{P}(J_k)$, respectively. Before we prove convergence, let us define some notations. Assume that any matrix $M$ has singular value decomposition $M=U\Sigma V^T$, where orthogonal matrices $U\in \mathbb{R}^{m\times m}$ and $V\in \mathbb{R}^{n\times n}$, and $\Sigma=\diag(\sigma_1,\cdots,\sigma_p)\in\mathbb{R}^{m\times n},~p=\min\{m,n\}$. We define $M^{\dag\frac{1}{2}}=V\Sigma^{\dag\frac{1}{2}}U^T$, where $\Sigma^{\dag\frac{1}{2}} = \diag(\sigma_1^{-\frac{1}{2}}, \cdots, \sigma_p^{-\frac{1}{2}}) \in\mathbb{R}^{n\times m}$. We define orthogonal projection
\begin{equation*}
P_1=A_{I_k,:}^\dag A_{I_k,:}, P_2=B_{:,J_k}B_{:,J_k}^\dag.
\end{equation*}
Then $P_1^T=P_1$, $P_1^2=P_1$, $P_2^T=P_2$, and $P_2^2=P_2$. The expectation of $P_1$ is
\begin{eqnarray*}
\mathbb{E}[P_1]  &=&\sum_{I_k\subseteq[m]}\mathbb{P}(I_k)A_{I_k,:}^\dag A_{I_k,:}\\
&=&\sum_{I_k\subseteq[m]}\mathbb{P}(I_k)A_{I_k,:}^T(A_{I_k,:}A_{I_k,:}^T)^\dag A_{I_k,:}\\
&=&\sum_{I_k\subseteq[m]}A_{I_k,:}^T(A_{I_k,:}A_{I_k,:}^T)^{\dag\tfrac{1}{2}}\mathbb{P}(I_k)  (A_{I_k,:}A_{I_k,:}^T)^{\dag\tfrac{1}{2}}A_{I_k,:}\\
&=&(\Delta A)^T(\Delta A),
\end{eqnarray*}
where $\Delta=\diag (\mathbb{P}(I_k)^{\frac{1}{2}}(A_{I_k,:}A_{I_k,:}^T)^{\dag\frac{1}{2}}, I_k\subseteq [m])$ is block diagonal matrix.
Similarly, the expectation of $P_2$ is
\begin{equation*}
\mathbb{E}[P_2]=(B\Gamma)(B\Gamma)^T,
\end{equation*}
where $\Gamma=\diag(\mathbb{P}(J_k)^{\frac{1}{2}}(B_{:,J_k}^TB_{:,J_k})^{\dag\frac{1}{2}}, J_k\subseteq [n])$ is block diagonal matrix.
We now analyze the convergence of the error $X_k-X_*$ for iterates of \Cref{alg:GRBK}. This result, stated in \Cref{Th:cov-1}, shows that the \Cref{alg:GRBK} will converge linearly to the solution of minimal Frobenius norm in expectation.

\begin{theorem}\label{Th:cov-1}
Let $X_*$ be the minimal Frobenius norm solution of $AXB=C$, and $X_k$ be the $k$th approximation of $X_*$ generated by the GRBK method. The expected norm of the error at the $k$th iteration satisfies
\begin{equation}\label{eq:cov-1}
\mathbb{E}\|X_k-X_*\|_F^2\leq\left(1-\sigma_{\min}^2(\Delta)\sigma_{\min}^2(\Gamma)\sigma_{\min}^2(A)\sigma_{\min}^2(B)\right)^k\|X_0-X_*\|_F^2.
\end{equation}
\end{theorem}
\begin{proof}
The step 3 of \Cref{alg:GRBK} can be rewritten as a simple fixed point formula
\begin{equation*}
X_{k+1}-X_*=X_k-X_*-A_{I_k,:}^\dag A_{I_k,:}(X_k-X_*)B_{:,J_k}B_{:,J_k}^\dag,
\end{equation*}
Since
\begin{eqnarray*}
&&\langle X_{k+1}-X_k, X_{k+1}-X_*\rangle_F\\
&&=\langle P_1(X_*-X_k)P_2, (X_k-X_*)-P_1(X_k-X_*)P_2 \rangle_F \\
&&=\langle P_1(X_*-X_k)P_2, (X_k-X_*)\rangle_F-\langle P_1(X_*-X_k)P_2, P_1(X_k-X_*)P_2 \rangle_F\\
&&=\langle P_1(X_*-X_k), (X_k-X_*)P_2\rangle_F-\langle P_1(X_*-X_k), (X_k-X_*)P_2 \rangle_F\\
&&=0.
\end{eqnarray*}
It follows that
\begin{equation*}
\|X_{k+1}-X_*\|_F^2=\|X_k-X_*\|_F^2-\|X_{k+1}-X_k\|_F^2.
\end{equation*}
Taking conditional expectations, we get
\begin{equation}\label{eq:p-1}
\mathbb{E}[\|X_{k+1}-X_*\|_F^2|X_k]=\|X_k-X_*\|_F^2-\mathbb{E}[\|X_{k+1}-X_k\|_F^2|X_k].
\end{equation}
Since
\begin{eqnarray*}
\|X_{k+1}-X_k\|_F^2 &=&\langle P_1(X_k-X_*)P_2, P_1(X_k-X_*)P_2 \rangle_F \\
&=&\langle P_1(X_k-X_*), (X_k-X_*)P_2 \rangle_F.
\end{eqnarray*}
Hence
\begin{eqnarray*}
\mathbb{E}[\|X_{k+1}-X_k\|_F^2|X_k] &=&
\langle \mathbb{E}[P_1](X_k-X_*), (X_k-X_*)\mathbb{E}[P_2] \rangle_F\\
&=&\left\langle (\Delta A)^T(\Delta A)(X_k-X_*), (X_k-X_*)(B\Gamma)(B\Gamma)^T \right\rangle_F\\
&=&\|(\Delta A)(X_k-X_*)(B\Gamma)\|_F^2.
\end{eqnarray*}
Using the fact
\begin{eqnarray*}
\|(\Delta A)(X_k-X_*)(B\Gamma)\|_F^2&=&\|[(B\Gamma)^T\otimes (\Delta A)]{\rm vec}(X_k-X_*)\|_2^2\\
&=&\|(\Gamma^T\otimes \Delta)(B^T\otimes A){\rm vec}(X_k-X_*)\|_2^2\\
&\geq&\sigma_{\min}^2(\Gamma^T\otimes \Delta)\|(B^T\otimes A){\rm vec}(X_k-X_*)\|_2^2\\
&\geq&\sigma_{\min}^2(\Gamma^T\otimes \Delta)\sigma_{\min}^2(B^T\otimes A)\|X_k-X_*\|_F^2\\
&=&\sigma_{\min}^2(\Delta)\sigma_{\min}^2(\Gamma)\sigma_{\min}^2(A)\sigma_{\min}^2(B)\|X_k-X_*\|_F^2,
\end{eqnarray*}we have
\begin{equation}\label{eq:p-2}
\mathbb{E}[\|X_{k+1}-X_k\|_F^2|X_k]  \geq \sigma_{\min}^2(\Delta)\sigma_{\min}^2(\Gamma)\sigma_{\min}^2(A)\sigma_{\min}^2(B)\|X_k-X_*\|_F^2.
\end{equation}
Thus, combining \eqref{eq:p-1} and \eqref{eq:p-2}, we can obtain an estimate an follows:
\begin{equation*}
\mathbb{E}[\|X_{k+1}-X_*\|_F^2|X_k]\leq\left(1-\sigma_{\min}^2(\Delta)\sigma_{\min}^2(\Gamma)\sigma_{\min}^2(A)\sigma_{\min}^2(B)\right)
\|X_k-X_*\|_F^2.
\end{equation*}
Taking the full expectation of both sides, we conclude that
\begin{equation*}
\mathbb{E}[\|X_{k+1}-X_*\|_F^2]\leq\left(1-\sigma_{\min}^2(\Delta)\sigma_{\min}^2(\Gamma)\sigma_{\min}^2(A)\sigma_{\min}^2(B)\right)
\mathbb{E}[\|X_k-X_*\|_F^2].
\end{equation*}
By induction, we complete the proof.
\end{proof}
\begin{remark}\label{Re:r-1}
If the sets $[m]$ and $[n]$ are partitioned by
\begin{equation*}
[m]=\{I_1, \cdots, I_s\}, ~[n]=\{J_1, \cdots, J_t\},
\end{equation*}
and the index sets $I_k\subseteq[m]$ and $J_k\subseteq[n]$ are selected according to probability distribution
\begin{equation*}
\mathbb{P}(I_k)=\tfrac{\|A_{I_k,:}\|_F^2}{\|A\|_F^2}~{\rm and}~\mathbb{P}(J_k)=\tfrac{\|B_{:,J_k}\|_F^2}{\|B\|_F^2},
\end{equation*}
respectively. Then
\begin{eqnarray*}
\sigma_{\min}(\Delta) &=&\min_{1\leq i\leq s}\tfrac{\|A_{I_i,:}\|_F}{\|A\|_F}\sigma_{\min}\left((A_{I_i,:}A_{I_i,:}^T)^{\dag\tfrac{1}{2}}\right)\\
&=&\min_{1\leq i\leq s}\tfrac{\|A_{I_i,:}\|_F}{\|A\|_F}\sigma_{\max}^{-1}(A_{I_i,:})\\
&=&\tfrac{1}{\|A\|_F}\left(\max_{1\leq i\leq s}\tfrac{\sigma_{\max}(A_{I_i,:})}{\|A_{I_i,:}\|_F}\right)^{-1}
\end{eqnarray*}
and
\begin{equation*}
\sigma_{\min}(\Gamma)=\tfrac{1}{\|B\|_F}\left(\max_{1\leq j\leq t}\tfrac{\sigma_{\max}(B_{:,J_j})}{\|B_{:,J_j}\|_F}\right)^{-1}.
\end{equation*}
Hence, the upper bound estimate of \eqref{eq:cov-1} becomes
\begin{equation}\label{eq:cov-1-1}
\mathbb{E}\|X_k-X_*\|_F^2\leq\left(1-\tfrac{\sigma_{\min}^2(A)}{\|A\|_F^2 \beta_{\max}^2(A)}
\tfrac{\sigma_{\min}^2(B)}{\|B\|_F^2 \beta_{\max}^2(B)}\right)^k\|X_0-X_*\|_F^2,
\end{equation}
where
\begin{equation}\label{eq:beta}
\beta_{\max}(A)=\max_{1\leq i\leq s}\tfrac{\sigma_{\max}(A_{I_i,:})}{\|A_{I_i,:}\|_F}{\rm~and~}
\beta_{\max}(B)=\max_{1\leq j\leq t}\tfrac{\sigma_{\max}(B_{:,J_j})}{\|B_{:,J_j}\|_F}.
\end{equation}
\end{remark}

Assume that $\beta_{\max}(A)=\tfrac{\sigma_{\max}(A_{I_{i_0},:})}{\|A_{I_{i_0},:}\|_F}$ and $\beta_{\max}(B)= \tfrac{\sigma_{\max} (B_{:,J_{j_0}})} {\|B_{:,J_{j_0}}\|_F}$. As
\begin{equation*}
\sigma_{\max}(A_{I_{i_0},:})>\sigma_{\min}(A), ~\|A\|_F>\|A_{I_{i_0},:}\|_F,
\end{equation*}
it holds that
\begin{equation*}
0<\tfrac{\sigma_{\min}^2(A)}{\|A\|_F^2 \beta_{\max}^2(A)}= \tfrac{\sigma_{\min}^2(A)\|A_{I_{i_0},:}\|_F^2}{\|A\|_F^2\sigma_{\max}^2(A_{I_{i_0},:})}<1.
\end{equation*}
Similarly, we have
\begin{equation*}
0<\tfrac{\sigma_{\min}^2(B)}{\|B\|_F^2 \beta_{\max}^2(B)}<1.
\end{equation*}
Thus, the convergence factor of inequality \eqref{eq:cov-1-1} is less than 1 and greater than 0. So the GRBK method converges to the minimal Frobenius norm solution of  $AXB = C$.

\begin{remark}\label{Re:r-2}
Now, let us consider the block index sets size $|I_k|=|J_k|=1$. In this case, the indices $i_k\in [m]$ and $j_k\in [n]$ are selected according to a probability distribution $\mathbb{P}(i_k)=\frac{\|A_{i_k,:}\|^2}{\|A\|_F^2}$ and $\mathbb{P}(j_k)=\frac{\|B_{:,j_k}\|^2}{\|B\|_F^2}$, respectively. Then, the update \eqref{eq:it-3} becomes
\begin{equation}\label{eq:it-3a}
X_{k+1}=X_k+\tfrac{A_{i_k,:}^T(C_{i_kj_k}-A_{i_k,:}X_kB_{:,j_k})B_{:,j_k}^T}{\|A_{i_k,:}\|^2\|B_{:,j_k}\|^2},
\end{equation}
which is called the global randomized Kaczmarz (GRK) method. Since
\begin{equation*}
\beta_{\max}(A)=\max_{1\leq i\leq m}\tfrac{\sigma_{\max}(A_{i,:})}{\|A_{i,:}\|_F}=1
~{\rm and}~
\beta_{\max}(B)=\max_{1\leq j\leq n}\tfrac{\sigma_{\max}(B_{:,j})}{\|B_{:,j}\|_F}=1.
\end{equation*}
Then, we get a linear convergence rate in the expectation of the form
\begin{equation}\label{eq:cov-1-2}
\mathbb{E}[\|X_k-X_*\|_F^2]\leq\left(1-\tfrac{\sigma_{\min}^2(A)}{\|A\|_F^2}\tfrac{\sigma_{\min}^2(B)}{\|B\|_F^2}\right)^k\|X_0-X_*\|_F^2.
\end{equation}
\end{remark}

Comparing \eqref{eq:cov-1-2} with the convergence rate \eqref{eq:cov-1-1}, since $\beta_{\max}(A)$ and $\beta_{\max}(B)$ are less than or equal to 1, we show the convergence factor of the GRBK method is smaller than that of the GRK method, which reveals that the GRBK method has a significant speed-up.

\section{The global randomized average block Kaczmarz algorithms}\label{Sec: GRABK}
In this section, we develop new variants of the global randomized block Kaczmarz algorithm for solving large-scale linear matrix equation $AXB=C$. In practice, the main drawback of \eqref{eq:it-3} is that each iteration is expensive and difficult to parallelize, since we need to compute the pseudoinverse of two submatrices. To take advantage of parallel computation and speed up the convergence of GRK, we consider a simple extension of the GRK method, where at each iteration multiple independent updates are computed in parallel and a weighted average of the updates is used. Specifically, we write the averaged GRK update:
\begin{equation}\label{eq:it-4}
X_{k+1}=X_k+\alpha_k\left(\sum_{i\in I_k,j\in J_k}u_i^{k}v_j^{k}\tfrac{A_{i,:}^T(C_{ij}-A_{i,:}X_kB_{:,j})B_{:,j}^T}{\|A_{i,:}\|^2\|B_{:,j}\|^2}\right),
\end{equation}
where the stepsize $\alpha_k\in(0,2)$ and the weights $u_i^{k},v_j^{k}\in [0,1]$ such that $\sum_{i\in I_k}u_i^{k}=1$ and $\sum_{j\in J_k}v_j^{k}=1$. The averaged GRK is detailed in \Cref{alg:GRABK}. If $I_k$ and $J_k$ are sets of size one, i.e. $I_k={i_k}$ and $J_k=\{j_k\}$, and $u_i^{k},v_j^{k}=1$ for $i=1,\cdots,m,j=1,\cdots,n$ and $k\geq0$, we recover the GRK method.

\begin{algorithm}[!htbp]
\caption{Global  Randomized Average Block Kaczmarz (GRABK)}
\label{alg:GRABK}
\hspace*{0.02in}{\bf Input:} {$A\in \mathbb{R}^{m \times p}$, $B\in \mathbb{R}^{q\times n}$, $C\in \mathbb{R}^{m \times n}$, $X_0\in \mathbb{R}^{p\times q}$, weights $u_i^{(k)}\geq0$, $v_j^{(k)}\geq0$, and stepsizes $\alpha\geq0$.}\\
\hspace*{0.02in}{\bf Output:} {Last iterate $X_{k+1}$.}
\begin{algorithmic}[1]
\For{$k=0,1,2,\cdots,$}
\State{Select a index set $I_k\subseteq[m]$ with probability $\mathbb{P}(I_k)>0$ such that $$\sum_{I_k\subseteq[m]}\mathbb{P}(I_k)=1;$$}
\State{Select a index set $J_k\subseteq[n]$ with probability $\mathbb{P}(J_k)>0$ such that $$\sum_{J_k\subseteq[n]}\mathbb{P}(J_k)=1;$$}
\State {Update $X_{k+1}=X_k+\alpha_k\left(\sum\limits_{i\in I_k,j\in J_k} u_i^{(k)}v_j^{(k)} \tfrac{A_{i,:}^T(C_{ij}-A_{i,:}X_kB_{:,j})B_{:,j}^T} {\|A_{i,:}\|^2\|B_{:,j}\|^2}\right) $.}
\EndFor
\end{algorithmic}
\end{algorithm}

Recall that the weights satisfy $0\leq u_i^{k},v_j^{k}\leq 1$ and $\sum_{i\in I_k}u_i^{k}=\sum_{j\in J_k}v_j^{k}=1$. Hence, we assume that the bounded of weights satisfy
\begin{equation*}
0<u_{\min}\leq u_i^{k}\leq u_{\max}<1~{\rm and}~0<v_{\min}\leq u_i^{k}\leq v_{\max}<1
\end{equation*}
for all $i\in I_k, j\in J_k$ and $k\geq0$. If the weights $u_i^{k}=\tfrac{\|A_{i,:}\|^2}{\|A_{I_k,:}\|_F^2}$ and $u_i^{k}=\tfrac{\|B_{:,j}\|^2}{\|B_{:,J_k}\|_F^2}$ for all $k\geq0$, we get the following compact update:
\begin{equation*}
X_{k+1}=X_k-\alpha_k\tfrac{A_{I_k,:}^T(A_{I_k,:}X_kB_{:,J_k}-C_{I_k,J_k})B_{:,J_k}^T}{\|A_{I_k,:}\|_F^2\|B_{:,J_k}\|_F^2}.
\end{equation*}
If the weights $u_i^{k}=\tfrac{1}{|I_k|}$ and $v_j^{k}=\tfrac{1}{|J_k|}$ for all $k\geq0$, we get the following compact update:
\begin{equation*}
X_{k+1}=X_k-\alpha_k\tfrac{A_{I_k,:}^TD_{I_k}^2(A_{I_k,:}X_kB_{:,J_k}-C_{I_k,J_k})D_{J_k}^2B_{:,J_k}^T}{|I_k||J_k|},
\end{equation*}
where the diagonal matrices
\begin{equation}\label{eq:D}
\begin{split}
&D_{I_k}=\diag(\|A_{i,:}\|^{-1}, ~i\in I_k) \in \mathbb{R}^{|I_k|\times|I_k|},\\
&D_{J_k}=\diag(\|B_{:,j}\|^{-1}, ~j\in J_k) \in \mathbb{R}^{|J_k|\times|J_k|}.
\end{split}
\end{equation}
In the rest of the section, we assume that the index set $I_k\subseteq[m]$ is chosen with probability $\mathbb{P}(I_k)>0$ such that $\sum_{I_k\subseteq[m]}\mathbb{P}(I_k)= 1$, and the index set $J_k\subseteq[n]$ is chosen with probability is $\mathbb{P}(J_k)>0$ such that $\sum_{J_k\subseteq[n]}\mathbb{P}(J_k )=1$. Let
\begin{equation*}
\tilde{A}_{I_k,:}=D_{I_k}A_{I_k,:} {\rm~and~} \tilde{B}_{:,J_k}=B_{:,J_k}D_{J_k}.
\end{equation*}
Then, the following equalities hold:
\begin{eqnarray*}
\mathbb{E}\left[(\tilde{A}_{I_k,:})^T\tilde{A}_{I_k,:}\right]&=&\sum_{I_k\subseteq[m]}\mathbb{P}(I_k)
\left[(D_{I_k}A_{I_k,:})^T(D_{I_k}A_{I_k,:})\right]\\
&=&(D_{_A}A)^T(D_{_A}A),
\end{eqnarray*}
and
\begin{eqnarray*}
\mathbb{E}\left[\tilde{B}_{:,J_k}(\tilde{B}_{:,J_k})^T\right]&=&\sum_{J_k\subseteq[n]}\mathbb{P}(J_k)
\left[(B_{:,J_k}D_{J_k})(B_{:,J_k}D_{J_k})^T\right]\\
&=&(BD_{_B})(BD_{_B})^T,
\end{eqnarray*}
where the block diagonal matrices
\begin{equation}\label{eq:D-1}
D_{_A}=\mathbb{P}(I_k)^{\frac{1}{2}}{\rm diag}(D_{I_k},I_k\subseteq[m]) {\rm~and~}
D_{_B}=\mathbb{P}(J_k)^{\frac{1}{2}}{\rm diag}(D_{J_k},J_k\subseteq[n]).
\end{equation}
Before we talk about stepsize $\alpha_k$,  let us define following notations:
\begin{equation}\label{eq: gamma}
\gamma_{\max}(A)=\max_{I_k\subseteq [m]}\sigma_{\max}(\tilde{A}_{I_k,:}),~
\gamma_{\max}(B)=\max_{J_k\subseteq [n] }\sigma_{\max}(\tilde{B}_{:,J_k}).
\end{equation}

By the iterative scheme \eqref{eq:it-4} and $C_{ij}=A_{i,:}X_*B_{:,j}$ for all $i\in [m]$ and $j\in[n]$, we have
\begin{equation*}
X_{k+1}-X_*=(X_k-X_*)-\alpha_k\left(\sum_{i\in I_k,j\in J_k}u_i^{k}v_j^{k}\tfrac{A_{i,:}^TA_{i,:}(X_k-X_*)B_{:,j}B_{:,j}^T}{\|A_{i,:}\|^2\|B_{:,j}\|^2}\right).
\end{equation*}
It follows that
\begin{equation}\label{eq:unfold}
\begin{split}
\|X_{k+1}-X_*\|_F^2&=\|X_k-X_*\|_F^2\\
&-2\alpha_k\left\langle \sum_{i\in I_k,j\in J_k}u_i^{k}v_j^{k} \tfrac{A_{i,:}^TA_{i,:}(X_k-X_*)B_{:,j}B_{:,j}^T}{\|A_{i,:}\|^2\|B_{:,j}\|^2},  X_k-X_*\right\rangle_F\\
&+\alpha_k^2\left\|\sum_{i\in I_k,j\in J_k}u_i^{k}v_j^{k} \tfrac{A_{i,:}^TA_{i,:}(X_k-X_*)B_{:,j}B_{:,j}^T}{\|A_{i,:}\|^2\|B_{:,j}\|^2}\right\|_F^2.
\end{split}
\end{equation}
In order to ensure strictly decrease of the sequence $\{\|X_k-X_*\|_F^2\}_{k=0}^{\infty}$, we need
\begin{equation*}
\scaleto{-2\alpha_k\left\langle \sum_{i\in I_k,j\in J_k}u_i^{k}v_j^{k} \tfrac{A_{i,:}^TA_{i,:}(X_k-X_*)B_{:,j}B_{:,j}^T}{\|A_{i,:}\|^2\|B_{:,j}\|^2},  X_k-X_*\right\rangle_F
+ \alpha_k^2\left\|\sum_{i\in I_k,j\in J_k}u_i^{k}v_j^{k}\tfrac{A_{i,:}^TA_{i,:}(X_k-X_*)B_{:,j}B_{:,j}^T}{\|A_{i,:}\|^2\|B_{:,j}\|^2}\right\|_F^2<0.}{30pt}
\end{equation*}
That is to say
\begin{equation}\label{Ieq:sizestep}
0<\alpha_k<2L_k,~L_k=
\tfrac{\left\langle \sum_{i\in I_k,j\in J_k}u_i^{k}v_j^{k} \frac{A_{i,:}^TA_{i,:}(X_k-X_*)B_{:,j}B_{:,j}^T}{\|A_{i,:}\|^2\|B_{:,j}\|^2},  X_k-X_*\right\rangle_F }
{\left\|\sum_{i\in I_k,j\in J_k}u_i^{k}v_j^{k}\frac{A_{i,:}^TA_{i,:}(X_k-X_*)B_{:,j}B_{:,j}^T}{\|A_{i,:}\|^2\|B_{:,j}\|^2}\right\|_F^2}.
\end{equation}
Next, we consider the global randomized average block Kaczmarz algorithm with constant stepsize and adaptive stepsize.

\subsection{The global randomized average block Kaczmarz algorithm with constant stepsize}
In this section, we study global randomized average block Kaczmarz algorithm with constant stepsize $\alpha_k=\alpha$ and the weights $u_i^{k}=u_i,v_j^{k}=v_j$. Hence, the update format \eqref{eq:it-4} becomes
\begin{equation*}
X_{k+1}=X_k+\alpha\left(\sum_{i\in I_k,j\in J_k}u_iv_j\tfrac{A_{i,:}^T(C_{ij}-A_{i,:}X_kB_{:,j})B_{:,j}^T}{\|A_{i,:}\|^2\|B_{:,j}\|^2}\right).
\end{equation*}
The weights satisfy $0<u_{\min}\leq u_i\leq u_{\max}<1,~0<v_{\min}\leq u_i\leq v_{\max}<1$ for all $i,~j$ and $\sum_{i\in I_k}u_i=\sum_{j\in J_k}v_j=1$. For each iteration, we want the step size to be consistent. Therefore, we need to find a lower bound on $L_k$. Here, we consider a constant stepsize of the form
\begin{equation}\label{eq:size}
\alpha=\eta \alpha_*
\end{equation}
for some $\eta\in(0,2)$,  where $\alpha_*=\tfrac{u_{\min}v_{\min}}{u_{\max}^2v_{\max}^2\gamma_{\max}^2(A)\gamma_{\max}^2(B)}$ is a lower bound on $L_k$, $\gamma_{\max}(A)$ and $\gamma_{\max}(B)$ are shown in equation \eqref{eq: gamma}. See proof of \Cref{Th:cov-2} for details. \Cref{Th:cov-2} proves the convergence rate of \Cref{alg:GRABK} with constant stepsize $\alpha$, which depends explicitly on the geometric properties of the matrices $A, B$ and submatrices $A_{I_k,:},B_{:,J_k}$.
\begin{theorem}\label{Th:cov-2}
Let $X_*$ be the minimal Frobenius norm solution of $AXB=C$, and $X_k$ be the $k$th approximation of $X_*$ generated by the GRABK method with the weights $0<u_{\min}\leq u_i\leq u_{\max}<1$ for all $i\in [m]$, and $0<v_{\min}\leq v_j\leq v_{\max}<1$ for all $j\in [n]$, and the stepsize $\alpha=\eta \alpha_*$ for some $\eta\in(0,2)$. Then, the expected norm of the error at the $k$th iteration satisfies
\begin{equation}\label{eq:cov-2}
\scaleto{
\mathbb{E}\|X_k-X_*\|_F^2\leq \left(1-\eta(2-\eta) \phi
\sigma_{\min}^2(D_{_A}) \sigma_{\min}^2(D_{_B}) \sigma_{\min}^2(A) \sigma_{\min}^2(B) \right)^k\|X_0-X_*\|_F^2,}{13pt}
\end{equation}
where $\phi=\tfrac{u_{\min}^2v_{\min}^2}{u_{\max}^2v_{\max}^2\gamma_{\max}^2(A) \gamma_{\max}^2(B)}$, and
the block diagonal matrices $D_{_A}$ and $D_{_B}$ are shown in equation \eqref{eq:D-1}.
\end{theorem}
\begin{proof}
Since
\begin{equation*}
\sum_{i\in I_k,j\in J_k}\tfrac{A_{i,:}^TA_{i,:}(X_k-X_*)B_{:,j}B_{:,j}^T}{\|A_{i,:}\|^2\|B_{:,j}\|^2}=
(\tilde{A}_{I_k,:})^T\tilde{A}_{I_k,:}(X_k-X_*)\tilde{B}_{:,J_k}(\tilde{B}_{:,J_k})^T.
\end{equation*}
For the second term of the equation \eqref{eq:unfold}, we have
\begin{eqnarray*}
&&\left\langle \sum_{i\in I_k,j\in J_k}u_iv_j \tfrac{A_{i,:}^TA_{i,:}(X_k-X_*)B_{:,j}B_{:,j}^T}{\|A_{i,:}\|^2\|B_{:,j}\|^2},  X_k-X_*\right\rangle_F \\
&&\geq u_{\min}v_{\min} \left\langle (\tilde{A}_{I_k,:})^T\tilde{A}_{I_k,:}(X_k-X_*) \tilde{B}_{:,J_k}(\tilde{B}_{:,J_k})^T,  X_k-X_*\right\rangle_F  \\
&&=u_{\min}v_{\min}\left\langle \tilde{A}_{I_k,:}(X_k-X_*) \tilde{B}_{:,J_k},
\tilde{A}_{I_k,:}(X_k-X_*)\tilde{B}_{:,J_k}\right\rangle_F\\
&&=u_{\min}v_{\min}\|\tilde{A}_{I_k,:}(X_k-X_*)\tilde{B}_{:,J_k}\|_F^2.
\end{eqnarray*}
For the third term of the equation \eqref{eq:unfold}, we have
\begin{eqnarray*}
&& \left\|\sum_{i\in I_k,j\in J_k}u_iv_j \tfrac{A_{i,:}^TA_{i,:}(X_k-X_*)B_{:,j}B_{:,j}^T}{\|A_{i,:}\|^2\|B_{:,j}\|^2}\right\|_F^2\\
&&\leq u_{\max}^2v_{\max}^2 \left\|(\tilde{A}_{I_k,:})^T\tilde{A}_{I_k,:}(X_k-X_*)\tilde{B}_{:,J_k}(\tilde{B}_{:,J_k})^T\right\|_F^2\\
&&=u_{\max}^2v_{\max}^2 \left\|[\tilde{B}_{:,J_k}\otimes (\tilde{A}_{I_k,:})^T]{\rm vec}[\tilde{A}_{I_k,:}(X_k-X_*)\tilde{B}_{:,J_k}]\right\|_2^2\\
&&\leq u_{\max}^2v_{\max}^2\sigma_{\max}^2\left(\tilde{B}_{:,J_k}\otimes (\tilde{A}_{I_k,:})^T\right)
\left\|{\rm vec}[\tilde{A}_{I_k,:}(X_k-X_*)\tilde{B}_{:,J_k}]\right\|_2^2\\
&&=u_{\max}^2v_{\max}^2\sigma_{\max}^2(\tilde{B}_{:,J_k})\sigma_{\max}^2(\tilde{A}_{I_k,:})\|\tilde{A}_{I_k,:} (X_k-X_*) \tilde{B}_{:,J_k}\|_F^2\\
&&\leq u_{\max}^2 v_{\max}^2 \gamma_{\max}^2(A) \gamma_{\max}^2(B) \|\tilde{A}_{I_k,:}(X_k-X_*)\tilde{B}_{:,J_k}\|_F^2.
\end{eqnarray*}
Hence
\begin{eqnarray*}
\|X_{k+1}-X_*\|_F^2&\leq&\|X_k-X_*\|_F^2-\big(2\alpha u_{\min}v_{\min}\\
&&-\alpha^2 u_{\max}^2 v_{\max}^2 \gamma_{\max}^2(A) \gamma_{\max}^2(B) \big) \|\tilde{A}_{I_k,:}(X_k-X_*)\tilde{B}_{:,J_k}\|_F^2.
\end{eqnarray*}
In order to ensure strictly decrease of the sequence $\{\|X_k-X_*\|_F^2\}_{k=0}^{\infty}$, we need
\begin{equation*}
2\alpha u_{\min}v_{\min}-\alpha^2 u_{\max}^2 v_{\max}^2 \gamma_{\max}^2(A) \gamma_{\max}^2(B)>0.
\end{equation*}
Hence, the stepsize
\begin{equation*}
0<\alpha<\tfrac{2u_{\min}v_{\min}}{u_{\max}^2v_{\max}^2\gamma_{\max}^2(A) \gamma_{\max}^2(B)}\leq2L_k,
\end{equation*}
and the optimal stepsize is obtained by maximizing
\begin{equation*}
2\alpha u_{\min}v_{\min}-\alpha^2 u_{\max}^2 v_{\max}^2 \gamma_{\max}^2(A) \gamma_{\max}^2(B)
\end{equation*}
with respect to $\alpha$,
which leads to $\alpha_*=\tfrac{u_{\min}v_{\min}}{u_{\max}^2v_{\max}^2\gamma_{\max}^2(A) \gamma_{\max}^2(B)}$.
Hence, taking stepsize $\alpha=\eta\alpha_*$ for some $\eta\in (0,2)$, we obtain
\begin{equation*}
\|X_{k+1}-X_*\|_F^2\leq\|X_k-X_*\|_F^2-\eta(2-\eta)\phi\|\tilde{A}_{I_k,:}(X_k-X_*)\tilde{B}_{:,J_k}\|_F^2.
\end{equation*}
Taking conditional expectations, we get
\begin{equation}\label{eq:p-1a}
\mathbb{E}[\|X_{k+1}-X_*\|_F^2|X_k]=\|X_k-X_*\|_F^2-\eta(2-\eta)\phi \mathbb{E}[\|\tilde{A}_{I_k,:}(X_k-X_*)\tilde{B}_{:,J_k}\|_F^2|X_k].
\end{equation}
We note that
\begin{eqnarray*}
&&\mathbb{E}[\|\tilde{A}_{I_k,:}(X_k-X_*)\tilde{B}_{:,J_k}\|_F^2|X_k]\\
&&=\left\langle \mathbb{E}[(\tilde{A}_{I_k,:})^T\tilde{A}_{I_k,:}](X_k-X_*),
(X_k-X_*)\mathbb{E}[\tilde{B}_{:,J_k}(\tilde{B}_{:,J_k})^T]\right\rangle_F\\
&&=\left\langle (D_{_A}A)^T(D_{_A}A)(X_k-X_*), (X_k-X_*)(BD_{_B})(BD_{_B})^T\right\rangle_F\\
&&=\|(D_{_A}A)(X_k-X_*)(BD_{_B})\|_F^2.
\end{eqnarray*}
Using the fact
\begin{eqnarray*}
\|(D_{_A}A)(X_k-X_*)(BD_{_B})\|_F^2&=&\|[(BD_{_B})^T\otimes (D_{_A}A)]{\rm vec}(X_k-X_*)\|_2^2\\
&=&\|(D_{_B}^T\otimes D_{_A})(B^T\otimes A){\rm vec}(X_k-X_*)\|_2^2\\
&\geq&\sigma_{\min}^2(D_{_B}^T\otimes D_{_A})\sigma_{\min}^2(B^T\otimes A)\|X_k-X_*\|_F^2\\
&=&\sigma_{\min}^2(D_{_A})\sigma_{\min}^2(A)\sigma_{\min}^2(D_{_B})\sigma_{\min}^2(B)\|X_k-X_*\|_F^2,
\end{eqnarray*}we have
\begin{equation}\label{eq:p-2a}
\scaleto{\mathbb{E}[\|\tilde{A}_{I_k,:}(X_k-X_*)\tilde{B}_{:,J_k}\|_F^2|X_k]\geq \sigma_{\min}^2(D_{_A}) \sigma_{\min}^2(A) \sigma_{\min}^2(D_{_B}) \sigma_{\min}^2(B) \|X_k-X_*\|_F^2.}{11pt}
\end{equation}
Thus, combining \eqref{eq:p-1a} and \eqref{eq:p-2a}, we can obtain an estimate an follows:
\begin{equation*}
\scaleto{
\mathbb{E}[\|X_{k+1}-X_*\|_F^2|X_k]\leq
\left(1-\eta(2-\eta)\phi \sigma_{\min}^2(D_{_A})\sigma_{\min}^2(A)\sigma_{\min}^2(D_{_B})\sigma_{\min}^2(B)\right)\|X_k-X_*\|_F^2.}{11pt}
\end{equation*}
Taking the full expectation of both sides, we conclude that
\begin{equation*}
\scaleto{
\mathbb{E}\|X_{k+1}-X_*\|_F^2\leq \left(1-\eta(2-\eta) \phi \sigma_{\min}^2(D_{_A})\sigma_{\min}^2(A)\sigma_{\min}^2(D_{_B})\sigma_{\min}^2(B) \right)\mathbb{E}\|X_k-X_*\|_F^2.}{11pt}
\end{equation*}
By induction, we complete the proof.
\end{proof}
\begin{remark}\label{Re:r-3}
Let $[m]=\{I_1, \cdots, I_s\}$ and $[n]=\{J_1, \cdots, J_t\}$ be partitions of  $[m]$ and $[n]$, respectively. The index sets $I_k\subseteq[m]$ and $J_k\subseteq[n]$ are selected according to probability distribution
\begin{equation*}
\mathbb{P}(I_k)=\tfrac{\|A_{I_k,:}\|_F^2}{\|A\|_F^2}~{\rm and}~\mathbb{P}(J_k)=\tfrac{\|B_{:,J_k}\|_F^2}{\|B\|_F^2},
\end{equation*}
respectively. In \Cref{alg:GRABK}, we take the weights $u_i^{k}=\tfrac{\|A_{i,:}\|_2^2}{\|A_{I_k,:}\|_F^2}$ and $v_j^{k}=\tfrac{\|B_{:,j}\|_2^2}{\|B_{:,J_k}\|_F^2}$ for all $k\geq0$, and the stepsize $\alpha=\eta \tfrac{1}{\beta_{\max}^2(A)\beta_{\max}^2(B)}$ for some $\eta\in(0,2)$. In this case, we have the following error estimate
\begin{equation}\label{eq:cov-2-1}
\mathbb{E}\|X_k-X_*\|_F^2\leq\left(1-\eta(2-\eta)\tfrac{\sigma_{\min}^2(A)}{\|A\|_F^2 \beta_{\max}^2(A)}\tfrac{\sigma_{\min}^2(B)}{\|B\|_F^2 \beta_{\max}^2(B)}\right)^k\|X_0-X_*\|_F^2,
\end{equation}
where  $\beta_{\max}(A)$ and $\beta_{\max}(B)$ are shown in equation \eqref{eq:beta}.
\end{remark}

\begin{remark}
By \Cref{Re:r-3}, we know that $0<\alpha<\tfrac{2}{\beta_{\max}^2(A)\beta_{\max}^2(B)}$ guarantees the convergence of the error $\mathbb{E}\|X_k-X_*\|_F^2$.  However, since the error estimate \eqref{eq:cov-2-1} usually is not sharp, the stepsize $\alpha$ satisfying $\tfrac{2}{\beta_{\max}^2(A)\beta_{\max}^2(B)}<\alpha<2\tfrac{\|A\|_F^2\|B\|_F^2}{\sigma_{\max}^2(A)\sigma_{\max}^2(B)}$ is also possible to result in convergence.
\end{remark}

\begin{remark}\label{Re:r-4}
Assume that the matrix $A$ is row normalized, i.e., $\|A_{i,:}\|=1$, and the matrix $B$ is column normalized, i.e., $\|B_{:,j}\|=1$. For index sets $I_k\subseteq[m]$ and $J_k\subseteq[n]$, the block sampling have the same size $|I_k|=\tau_1$ and $|J_k|=\tau_2$ for all $k\geq0$.
In this case, we have $\mathbb{P}(I_k)=\frac{\tau_1}{m}$ and $\mathbb{P}(J_k)=\frac{\tau_2}{n}$. Let us consider the particular choices $\eta=1$, the weights $u_i=\tfrac{1}{\tau_1}$ and $v_i=\tfrac{1}{\tau_2}$. Since
\begin{eqnarray*}
\gamma_{\max}^2(A)&=&\max_{I_k\subseteq [m]}\sigma_{\max}^2(A_{I_k,:})=\tau_1\beta_{\max}^2(A),\\
\gamma_{\max}^2(B)&=&\max_{J_k\subseteq [n]}\sigma_{\max}^2(B_{:,J_k})=\tau_2\beta_{\max}^2(B).
\end{eqnarray*}
Then, the convergence rate \eqref{eq:cov-2} and \eqref{eq:cov-2-1} become
\begin{equation}\label{eq:cov-2-2}
\mathbb{E}\|X_k-X_*\|_F^2\leq\left(1-\tfrac{\tau_1\tau_2}{\gamma_{\max}^2(A)\gamma_{\max}^2(B)}
\tfrac{\sigma_{\min}^2(A)}{m}\tfrac{\sigma_{\min}^2(B)}{n} \right)^k\|X_0-X_*\|_F^2,
\end{equation}
where $\gamma_{\max}(A)$ and $\gamma_{\max}(B)$ are shown in equation \eqref{eq: gamma}.
\end{remark}

Comparing \eqref{eq:cov-2-2} with the convergence rate \eqref{eq:cov-1-2}, since $\tfrac{\tau_1}{\gamma_{\max}^2(A)}$ and $\tfrac{\tau_2}{\gamma_{\max}^2(B)}$ greater than or equal to 1, we show the convergence factor of the GRABK method is smaller than that of the GRK method, which reveals that the GRABK method has a significant speed-up.

\subsection{The global randomized average block Kaczmarz algorithm with adaptive stepsize}
Since the GRABK method involved a stepsize $\alpha$ depending on the geometric properties of the $A, B$ and submatrices $A_{I_k,:}, B_{:, J_k}$, which may be difficult to compute in large-scale matrix equations. Next, we design a randomized average block Kaczmarz method with adaptive stepsize, which does not require the computation of $\gamma_{\max}^2(A)$, $\gamma_{\max}^2(B)$, $A_{I,:}^\dag$, and $B_{:,J}^\dag$. To simplify the notation, we define $\hat{u}_i^{k}=\frac{u_i^{k}}{\|A_{i,:}\|^2}$ and $\hat{v}_j^{k}=\frac{v_j^{k}}{\|B_{:,j}\|^2}$. Thus, the iterative formula \eqref{eq:it-4} becomes
\begin{equation*}
X_{k+1}=X_k+\alpha_k\left(\sum_{i\in I_k,j\in J_k}\hat{u}_i^{k}\hat{v}_j^{k}A_{i,:}^T(C_{ij}-A_{i,:}X_kB_{:,j})B_{:,j}^T\right).
\end{equation*}
By \eqref{Ieq:sizestep}, for each iteration, we consider the adaptive stepsize of the form
\begin{equation}\label{eq:stepsize}
\alpha_k=\eta L_k,~{\rm where}~L_k=\tfrac{\sum_{i\in I_k,j\in J_k}\hat{u}_i^{k}\hat{v}_j^{k}(C_{ij}-A_{i,:}X_kB_{:,j})^2}
{\left\|\sum_{i\in I_k,j\in J_k}\hat{u}_i^{k}\hat{v}_j^{k}A_{i,:}^T(C_{ij}-A_{i,:}X_kB_{:,j})B_{:,j}^T\right\|_F^2}
\end{equation}
for some $\eta\in(0,2)$. The convergence of the GRABK method with adaptive stepsize $\alpha_k$ is guaranteed by the \Cref{Th:cov-3}. The convergence rate of the GRABK method with adaptive stepsize $\alpha_k$ depends explicitly on the geometric properties of the matrices $A, B$ and submatrices $A_{I_k,:}, B_{:, J_k}$.
\begin{theorem}\label{Th:cov-3}
Let $X_*$ be the minimal Frobenius norm solution of $AXB=C$, and $X_k$ be the $k$th approximation of $X_*$ generated by the GRABK method with the weights $0<u_{\min}\leq u_i^{k}\leq u_{\max}<1$ for all $i\in [m]$, and $0<v_{\min}\leq v_j^{k}\leq v_{\max}<1$ for all $j\in [n]$, and the stepsize $\alpha_k=\eta L_k$ for some $\eta\in(0,2)$. Then, the expected norm of the error at the $k$th iteration satisfies
\begin{equation}\label{eq:cov-3}
\scaleto{
\mathbb{E}\|X_k-X_*\|_F^2\leq \left(1-\eta(2-\eta)\psi
 \sigma_{\min}^2(D_{_A}) \sigma_{\min}^2(D_{_B}) \sigma_{\min}^2(A) \sigma_{\min}^2(B) \right)^k\|X_0-X_*\|_F^2,}{13pt}
\end{equation}
where $\psi=\tfrac{u_{\min}v_{\min}}{u_{\max}v_{\max}\gamma_{\max}^2(A) \gamma_{\max}^2(B)}$, and
the block diagonal matrices $D_{_A}$ and $D_{_B}$ are shown in equation \eqref{eq:D-1}.
\end{theorem}
\begin{proof}
Let $\hat{A}_{i,:}=(\hat{u}_i^{k})^\frac{1}{2}A_{i,:}$ and $\hat{B}_{:,j}=B_{:,j}(\hat{v}_j^{k})^\frac{1}{2}$. Hence
\begin{equation*}
X_{k+1}=X_k+\alpha_k\left(\sum_{i\in I_k,j\in J_k}(\hat{A}_{i,:})^T \hat{A}_{i,:} (X_*-X_k) \hat{B}_{:,j}(\hat{B}_{:,j})^T\right),
\end{equation*}
and
\begin{equation*}
L_k=\tfrac{\sum_{i\in I_k,j\in J_k}[\hat{A}_{i,:}(X_k-X_*)\hat{B}_{:,j}]^2}
{\left\|\sum_{i\in I_k,j\in J_k}(\hat{A}_{i,:})^T \hat{A}_{i,:}(X_*-X_k)\hat{B}_{:,j}(\hat{B}_{:,j})^T\right\|_F^2}.
\end{equation*}
Using that
\begin{equation*}
\left\langle X_k-X_*, (\hat{A}_{i,:})^T \hat{A}_{i,:} (X_*-X_k) \hat{B}_{:,j}(\hat{B}_{:,j})^T\right\rangle_F
=-[\hat{A}_{i,:}(X_k-X_*)\hat{B}_{:,j}]^2,
\end{equation*}
we get
\begin{eqnarray*}
\|X_{k+1}-X_*\|_F^2&=&\|X_k-X_*\|_F^2-2\alpha_k \sum_{i\in I_k,j\in J_k}[\hat{A}_{i,:}(X_k-X_*)\hat{B}_{:,j}]^2 \\
&&+\alpha_k^2\left\|\sum_{i\in I_k,j\in J_k}(\hat{A}_{i,:})^T \hat{A}_{i,:} (X_*-X_k) \hat{B}_{:,j}(\hat{B}_{:,j})^T\right\|_F^2\\
&=&\|X_k-X_*\|_F^2-2\eta L_k\sum_{i\in I_k,j\in J_k}[\hat{A}_{i,:}(X_k-X_*)\hat{B}_{:,j}]^2\\
&&+\eta^2L_k^2\left\|\sum_{i\in I_k,j\in J_k}(\hat{A}_{i,:})^T \hat{A}_{i,:} (X_*-X_k) \hat{B}_{:,j}(\hat{B}_{:,j})^T\right\|_F^2\\
&=&\|X_k-X_*\|_F^2-\eta(2-\eta)L_k\sum_{i\in I_k,j\in J_k}[\hat{A}_{i,:}(X_k-X_*)\hat{B}_{:,j}]^2.
\end{eqnarray*}
Since
\begin{eqnarray*}
&&\left\|\sum_{i\in I_k,j\in J_k}(\hat{A}_{i,:})^T \hat{A}_{i,:}(X_*-X_k)\hat{B}_{:,j}(\hat{B}_{:,j})^T\right\|_F^2\\
&&=\left\|(\hat{A}_{I_k,:})^T\hat{A}_{I_k,:} (X_k-X_*) \hat{B}_{:,J_k}(\hat{B}_{:,J_k})^T\right\|_F^2 \\
&&=\left\|[\hat{B}_{:,J_k}\otimes(\hat{A}_{I_k,:})^T]{\rm vec}[\hat{A}_{I_k,:} (X_k-X_*) \hat{B}_{:,J_k}]\right\|_2^2 \\
&&\leq \sigma_{\max}^2\left(\hat{B}_{:,J_k}\otimes(\hat{A}_{I_k,:})^T\right) \left\|\hat{A}_{I_k,:} (X_k-X_*) \hat{B}_{:,J_k}\right\|_F^2 \\
&&\leq \sigma_{\max}^2(\hat{A}_{I_k,:})\sigma_{\max}^2(\hat{B}_{:,J_k}) \left\|\hat{A}_{I_k,:} (X_k-X_*) \hat{B}_{:,J_k}\right\|_F^2,
\end{eqnarray*}
where $\hat{A}_{I_k,:}=U_{I_k}\tilde{A}_{I_k,:}$ and $\hat{B}_{:,J_k}=\tilde{B}_{:,J_k}V_{J_k}$,
and the diagonal matrices $U_{I_k}=\diag ((u_i^{k})^{\frac{1}{2}},~i\in I_k)$ and $V_{J_k}=\diag((v_j^{k})^{\frac{1}{2}},~j\in J_k)$, $D_{I_k}$ and $D_{J_k}$ are shown in equation \eqref{eq:D}. In addition
\begin{equation*}
\sum_{i\in I_k,j\in J_k}[\hat{A}_{i,:}(X_k-X_*)\hat{B}_{:,j}]^2=\left\|\hat{A}_{I_k,:} (X_k-X_*) \hat{B}_{:,J_k} \right\|_F^2.
\end{equation*}
Furthermore
\begin{equation*}
\sigma_{\max}^2(\hat{A}_{I_k,:})\leq u_{\max}\sigma_{\max}^2(\tilde{A}_{I_k,:})\leq u_{\max}\gamma_{\max}^2(A)
\end{equation*}
and
\begin{equation*}
\sigma_{\max}^2(\hat{B}_{:,J_k})\leq v_{\max}\sigma_{\max}^2(\tilde{B}_{:,J_k})\leq v_{\max}\gamma_{\max}^2(B).
\end{equation*}
Hence
\begin{equation*}
L_k\geq\tfrac{1}{\sigma_{\max}^2(\hat{A}_{I_k,:})\sigma_{\max}^2(\hat{B}_{:,J_k})} \geq\tfrac{1}{u_{\max}v_{\max}\gamma_{\max}^2(A)\gamma_{\max}^2(B)}.
\end{equation*}
Therefore
\begin{eqnarray*}
\|X_{k+1}-X_*\|_F^2&\leq&\|X_k-X_*\|_F^2-\tfrac{\eta(2-\eta)}{u_{\max}v_{\max}\gamma_{\max}^2(A)\gamma_{\max}^2(B)} \left\|\hat{A}_{I_k,:} (X_k-X_*) \hat{B}_{:,J_k} \right\|_F^2\\
&\leq&\|X_k-X_*\|_F^2-\tfrac{\eta(2-\eta)u_{\min}v_{\min}}{u_{\max}v_{\max}\gamma_{\max}^2(A)\gamma_{\max}^2(B)} \left\|\tilde{A}_{I_k,:} (X_k-X_*) \tilde{B}_{:,J_k} \right\|_F^2.
\end{eqnarray*}
Taking the conditional expectation and using \eqref{eq:p-2a}, we get
\begin{equation*}
\scaleto{
\mathbb{E}[\|X_{k+1}-X_*\|_F^2|X_k] \leq
\left(1-\eta(2-\eta)\psi\sigma_{\min}^2(D_{_A})\sigma_{\min}^2(A)\sigma_{\min}^2(D_{_B})\sigma_{\min}^2(B)\right)\|X_k-X_*\|_F^2.}{11pt}
\end{equation*}
Taking the full expectation of both sides, we conclude that
\begin{equation*}
\scaleto{
\mathbb{E}\|X_{k+1}-X_*\|_F^2\leq
\left(1-\eta(2-\eta)\psi \sigma_{\min}^2(D_{_A})\sigma_{\min}^2(A)\sigma_{\min}^2(D_{_B})\sigma_{\min}^2(B)\right)\mathbb{E}\|X_k-X_*\|_F^2.}{11pt}
\end{equation*}
By induction, we complete the proof.
\end{proof}
\begin{remark}\label{Re:r-5}
Under the conditions of \Cref{Re:r-3}, we take the adaptive stepsize $\alpha_k=\eta L_k$ for some $\eta\in(0,2)$ and the weights $u_i^{k}=\tfrac{\|A_{i,:}\|_2^2}{\|A_{I_k,:}\|_F^2}$ and $v_j^{k}=\tfrac{\|B_{:,j}\|_2^2}{\|B_{:,J_k}\|_F^2}$ for all $k\geq0$ in \Cref{alg:GRABK}. In this case, $L_k\geq\tfrac{1}{\beta_{\max}^2(A)\beta_{\max}^2(B)}$, we have the following error estimate
\begin{equation}\label{eq:cov-3-1}
\mathbb{E}\|X_k-X_*\|_F^2\leq\left(1-\eta(2-\eta)\tfrac{\sigma_{\min}^2(A)}{\|A\|_F^2 \beta_{\max}^2(A)}\tfrac{\sigma_{\min}^2(B)}{\|B\|_F^2 \beta_{\max}^2(B)}\right)^k\|X_0-X_*\|_F^2.
\end{equation}
\end{remark}
\begin{remark}Under the assumptions and conditions of \Cref{Re:r-4}, the convergence rate \eqref{eq:cov-3} becomes
\begin{equation}\label{eq:cov-3-2}
\mathbb{E}\|X_k-X_*\|_F^2\leq\left(1-\tfrac{\tau_1\tau_2}{\gamma_{\max}^2(A)\gamma_{\max}^2(B)}
\tfrac{\sigma_{\min}^2(A)}{m}\tfrac{\sigma_{\min}^2(B)}{n} \right)^k\|X_0-X_*\|_F^2.
\end{equation}
We observe that this convergence rate is the same as \eqref{eq:cov-2-2}. However, the \Cref{alg:GRABK} with adaptive stepsize has more chances to accelerate.
\end{remark}

\begin{remark}
Let us consider the particular choices $\eta=1$, the convergence rate \eqref{eq:cov-2-1}and \eqref{eq:cov-3-1}become \eqref{eq:cov-1-1}, this implies that the GRBK and GRABK methods have the same convergence rate.  However, for solving large-scale matrix equations, the GRABK method can run in parallel and does not need to compute pseudoinverses. As a result, the GRABK method requires fewer CPU times than the GRBK method.
\end{remark}

There is a tight connection between the constant stepsize \eqref{eq:size} and the adaptive stepsize \eqref{eq:stepsize}.
In the proofs of \Cref{Th:cov-2,Th:cov-3}, the lower bounds of $L_k$ are given respectively. Since $\tfrac{u_{\min}}{u_{\max}}\leq 1$ and $\tfrac{v_{\min}}{v_{\max}} \leq 1$, it holds that
\begin{equation*}
L_k\geq\tfrac{1}{u_{\max}v_{\max}\gamma_{\max}^2(A)\gamma_{\max}^2(B)}\geq \tfrac{u_{\min}v_{\min}}{u_{\max}^2v_{\max}^2\gamma_{\max}^2(A)\gamma_{\max}^2(B)}=\alpha_*.
\end{equation*}
Hence, the adaptive stepsize \eqref{eq:stepsize} can be viewed as a practical approximation of the constant stepsize \eqref{eq:size}. However, the GRABK method with adaptive stepsize \eqref{eq:size} has more chances to accelerate, since the adaptive stepsize is, in general, larger than the constant stepsize counterpart.

\section{Numerical results}\label{Sec:NR}
In this section, we investigate the computational behavior of GRABK for solving various matrix equations, and compare GRABK with RK \cite{Strohmer2009}, i.e., the randomized Kaczmarz method is applied to the linear system $(B^T \otimes A) {\rm vec}(X) = {\rm vec} (C)$, and RBCD \cite{Du22} methods. All experiments are carried out using MATLAB (version R2017b) on a personal computer with 1.60 GHz central processing unit (Intel(R) Core(TM) i5-8265U CPU), 8.00 GB memory, and Windows operating system (64 bit Windows 10). We divided our tests into three broad categories: synthetic dense data, real-world sparse data, and an application to image restoration.

To construct a matrix equation, we set $C=AXB$, where $X$ is a random matrix with entries generated from a standard normal distribution. All computations are started from the initial guess $X_0 = 0$, and terminated once the \emph{relative error} (RE) of the solution, defined by
\begin{equation*}
{\rm RE}=\tfrac{\|X_k-X_*\|_F^2}{\|X_*\|_F^2}
\end{equation*}
at the current iterate $x_k$, satisfies $\text{RE} < 10^{-6}$ or exceeded maximum iteration, where $X_*=A^{\dag}CB^{\dag}$. We report the average number of iterations (IT) and the average CPU times in seconds (CPU) for 10 times repeated runs. We consider the following GRBK and GRABK variants:
\begin{itemize}
    \item GRBK with partition sampling and as in \Cref{Re:r-1}.
    \item GRK: GRBK with the block index sets size $|I_k | = |J_k | = 1$ as in \Cref{Re:r-2}.
	\item GRABK-c: GRABK with partition sampling and constant stepsize $\alpha=\tfrac{\eta}{\beta_{\max}^2(A)\beta_{\max}^2(B)}$ for some $\eta\in(0,2)$ as in \Cref{Re:r-3}.
	\item GRABK-a: GRABK with partition sampling and adaptive stepsize $\alpha_k=\eta L_k$ for some $\eta\in(0,2)$ as in \Cref{Re:r-5}.
\end{itemize}

For the block methods, we assume that $[m]=\{I_1, \cdots, I_s\}$ and $[n]=\{J_1, \cdots, J_t\}$ respectively are partitions of  $[m]$ and $[n]$, and the block sampling have the same size $|I_k|=\tau_1$ and $|J_k|=\tau_2$, where
\begin{equation*}
I_i=
\begin{cases}
\{(i-1)\tau_1+1, (i-1)\tau_1+2, \cdots, i\tau_1\}, &i=1,2,\cdots,s-1, \\
\{(s-1)\tau_1+1, (s-1)\tau_1+2, \cdots, m\}, &i=s,
\end{cases}
\end{equation*}
and
\begin{equation*}
J_j=
\begin{cases}
\{(j-1)\tau_2+1, (j-1)\tau_2+2, \cdots, j\tau_2\}, &j=1,2,\cdots,t-1, \\
\{(t-1)\tau_2+1, (t-1)\tau_2+2, \cdots, n\}, &j=t.
\end{cases}
\end{equation*}
\subsection{Synthetic dense data}
Synthetic dense data for this test is generated as follows:
\begin{itemize}
  \item Type I: For given $m,p$, and $r_1={\rm rank}(A)$, we construct a matrix $A$ by
  \begin{equation*}
  A=U_1D_1V_1^T,
  \end{equation*}
  where $U_1\in \mathbb{R}^{m\times r_1}$ and $V_1\in \mathbb{R}^{p\times r_1}$ are orthogonal columns matrices. The entries of $U_1$ and $V_1$ are generated from a standard normal distribution, and then columns are orthogonalization, i.e.,
  \begin{equation*}
  [U_1,\sim]={\rm qr(randn}(m,r_1),0),[V_1,\sim]={\rm qr(randn}(p,r_1),0).
  \end{equation*}
  The matrix $D_1$ is an $r_1\times r_1$ diagonal matrix whose diagonal entries are uniformly distribution numbers in $(1,2)$, i.e.,
  \begin{equation*}
  D_1=\diag(1+{\rm rand}(r_1,1)).
  \end{equation*}
  Similarly, for given $q,n$, and $r_2={\rm rank}(B)$, we construct a matrix $B$ by
  \begin{equation*}
  B=U_2D_2V_2^T,
  \end{equation*}
  where $U_2\in \mathbb{R}^{q\times r_2}$ and $V_1\in \mathbb{R}^{n\times r_2}$ are orthogonal columns matrices, and the matrix $D_2$ is an $r_2\times r_2$ diagonal matrix.
  \item Type II: For given $m,p,q,n$, the entries of $A$ and $B$ are generated from a standard normal distribution, i.e.,
  \begin{equation*}
  A={\rm randn}(m,p),~B={\rm randn}(q,n).
  \end{equation*}
\end{itemize}

\begin{figure}[!htb]
	\centering
	\subfigure[Type I]{\includegraphics[width=5in]{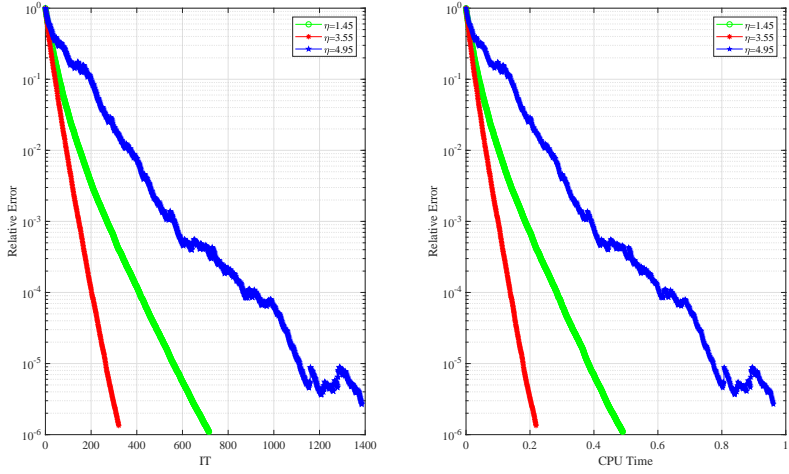}}
	\subfigure[Type II]{\includegraphics[width=5in]{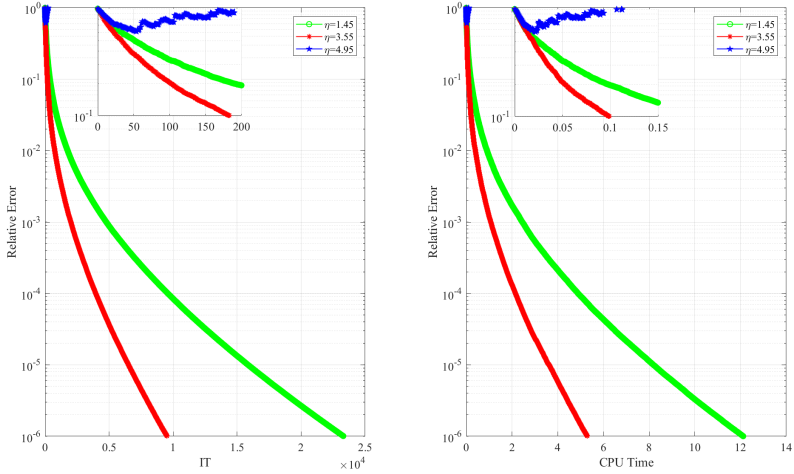}}
	\caption{The relative error of GRABK-c with block size $\tau_1=\tau_2=50$ and different stepsizes $\alpha=\tfrac{\eta}{\beta_{\max}^2(A)\beta_{\max}^2(B)}$ for two matrix equation.}
	\label{fig1}
\end{figure}

\begin{figure}[!htb]
	\centering
	\subfigure[Type I]{\includegraphics[width=5in]{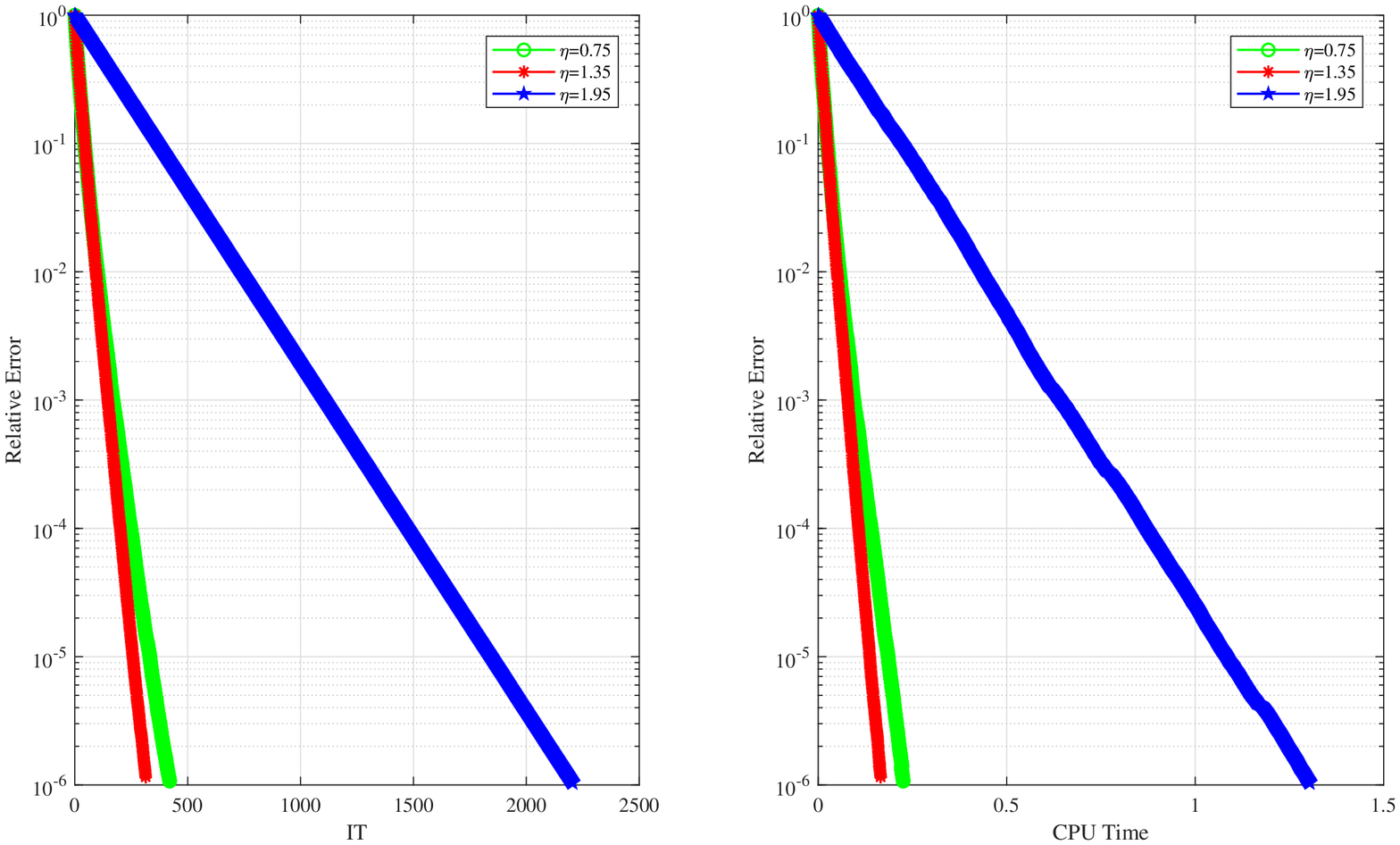}}
	\subfigure[Type II]{\includegraphics[width=5in]{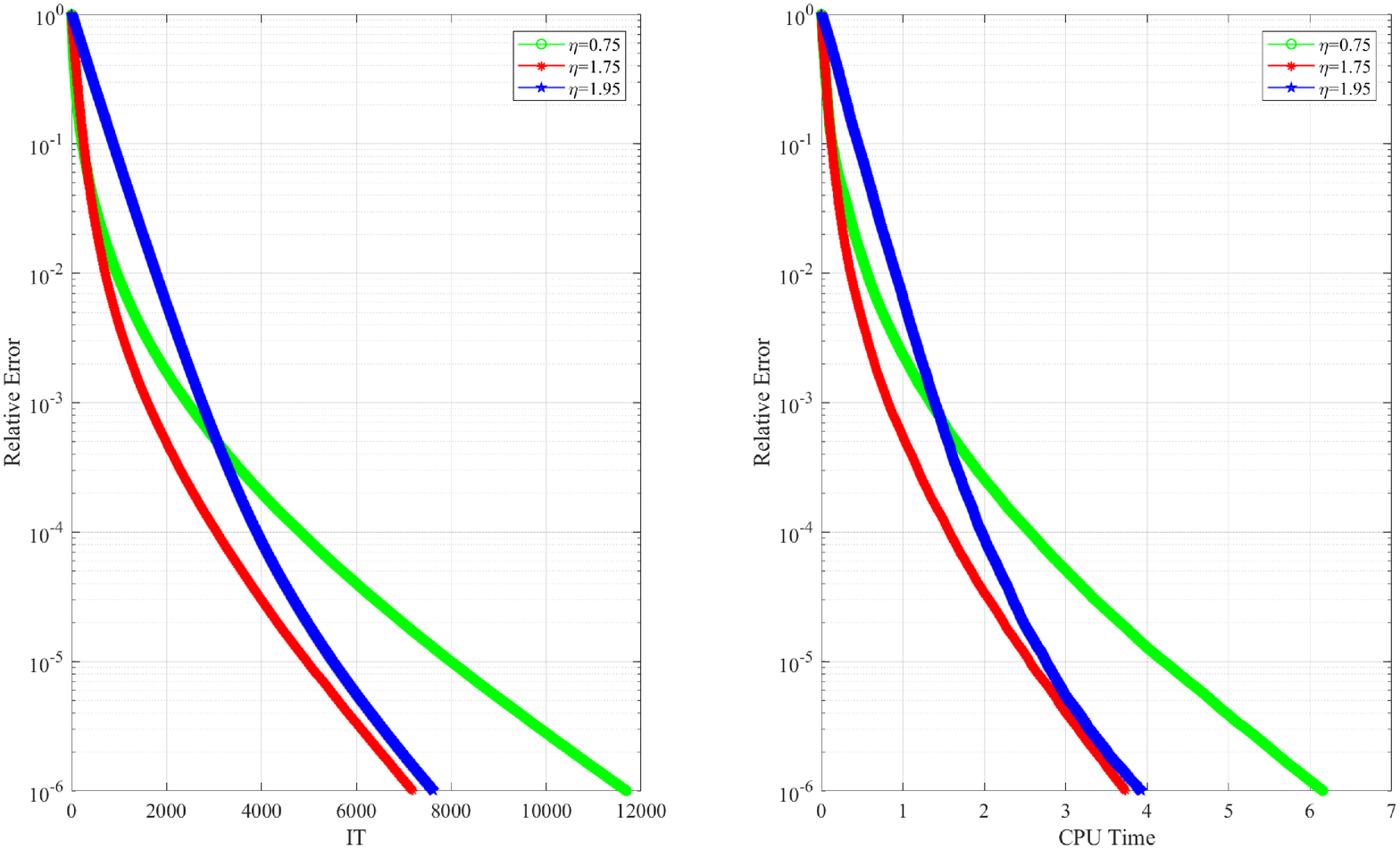}}
	\caption{The relative error of GRABK-a with block size $\tau_1=\tau_2=50$ and different stepsizes $\alpha_k=\eta L_k$ for two matrix equation.}
	\label{fig2}
\end{figure}

In \Cref{fig1}, we plot the relative error of GRABK-c with a fixed block size $\tau_1=\tau_2=50$ and different stepsizes $\alpha=\tfrac{\eta}{\beta_{\max}^2(A)\beta_{\max}^2(B)}$ for two matrix equation with Type I ($A=U_1D_1V_1^T$ with $m=500, p=250,r_1=150$ and $B=U_2D_2V_2^T$ with $n=500, q=250, r_2=150$) and Type II ($A={\rm randn}(500,250)$ and $B={\rm randn}(250,500)$).   Similarly, in \Cref{fig2}, we plot the relative error of GRABK-a with a fixed block size $\tau_1=\tau_2=50$ and different stepsizes $\alpha=\eta L_k$ for two matrix equation with Type I ($A=U_1D_1V_1^T$ with $m=500, p=250,r_1=150$ and $B=U_2D_2V_2^T$ with $n=500, q=250, r_2=150$) and Type II ($A={\rm randn}(500,250)$ and $B={\rm randn}(250,500)$). From \Cref{fig1,fig2}, we observed that the convergence rate of the GRABK-c and GRABK-a becomes faster as the stepsize increases, and then becomes slower after reaching the fastest rate, or even does not converge. Hence, in order to ensure convergence and weigh the convergence rate, we choose the stepsize $\alpha_k=\tfrac{1.95}{\beta_{\max}^2(A)\beta_{\max}^2(B)}$ in GRABK-c and $\alpha=L_k$ in GRABK-a.

\begin{table}[!htb]
\centering
	\caption{The average IT and CPU of RK, GRK, RBCD, GRBK, GRABK-c, and GRABK-a for matrix equations with Type I.}
	\scalebox{0.6}{
\begin{tabular}{lllllllllllllll}
\toprule
$m$    & $p$    & $r_1$& $\tau_1$ & $q$    & $n$    & $r_2$ & $\tau_2$ &     & RK      & GRK     & RBCD     & GRBK    & GRABK-c & GRABK-a \\
\midrule
50   & 20   & 10         & 10   & 20   & 50   & 20         & 10   & IT  & 5386.6  & 5186.3  & 473.1    & 26.0    & 185.7    & 123.9    \\
     &      &            &      &      &      &            &      & CPU & 0.2052  & 0.1430  & 0.0168   & 0.0047  & 0.0071   & 0.0053   \\
\midrule
50   & 20   & 10         & 10   & 20   & 50   & 10         & 5    & IT  & 2275.6  & 2353.4  & -        & 21.3    & 206.1    & 108.3    \\
     &      &            &      &      &      &            &      & CPU & 0.0842  & 0.0631  & -       & 0.0029  & 0.0072   & 0.0038   \\
\midrule
100  & 40   & 20         & 20   & 40   & 100  & 40         & 20   & IT  & 19789.9 & 19764.0 & 1045.2   & 24.3    & 216.9    & 131.0    \\
     &      &            &      &      &      &            &      & CPU & 2.8006  & 0.5695  & 0.0753   & 0.0094  & 0.0111   & 0.0072   \\
\midrule
100  & 40   & 20         & 20   & 40   & 100  & 20         & 10   & IT  & 9289.4  & 9577.2  & -        & 23.5    & 213.8    & 118.3    \\
     &      &            &      &      &      &            &      & CPU & 1.3268  & 0.2785  & -        & 0.0066  & 0.0106   & 0.0057   \\
\midrule
500  & 100  & 50         & 100  & 40   & 100  & 40         & 20   & IT  & 48986.7 & 48786.7 & 1169.8   & 28.4    & 178.3    & 91.0     \\
     &      &            &      &      &      &            &      & CPU & 28.6252 & 1.9705  & 0.4446   & 0.0435  & 0.0286   & 0.0127   \\
\midrule
500  & 100  & 50         & 100  & 40   & 100  & 20         & 10   & IT  & 24653.3 & 24326.1 & -        & 25.7    & 171.0    & 92.9     \\
     &      &            &      &      &      &            &      & CPU & 13.9705 & 1.0288  & -        & 0.0345  & 0.0274   & 0.0121   \\
\midrule
500  & 100  & 50         & 100  & 100  & 500  & 100        & 50   & IT  &$>$       &$>$       & 2569.3   & 24.7    & 172.6    & 88.5     \\
     &      &            &      &      &      &            &      & CPU &$>$       &$>$       & 4.49404 & 0.0632 & 0.04739 & 0.02415 \\
\midrule
500  & 100  & 50         & 100  & 100  & 500  & 50         & 25   & IT  & $>$       & $>$& -        & 23.0    & 197.7    & 89.9     \\
     &      &            &      &      &      &            &      & CPU & $>$       & $>$       & -        & 0.0429  & 0.0461   & 0.0200   \\
\midrule
1000 & 200  & 100        & 200  & 100  & 500  & 100        & 50   & IT  &$>$       & $>$       & 2348.1   & 24.4    & 162.7    & 87.4     \\
     &      &            &      &      &      &            &      & CPU &$>$       &$>$       & 13.4265  & 0.1847  & 0.1071   & 0.0538   \\
\midrule
1000 & 200  & 100        & 200  & 100  & 500  & 50         & 25   & IT  & $>$       &$>$       &-        & 23.6    & 217.0    & 93.7     \\
     &      &            &      &      &      &            &      & CPU & $>$       & $>$       & -        & 0.1454  & 0.1048   & 0.0453   \\
\midrule
1000 & 200  & 100        & 200  & 200  & 1000 & 200        & 100  & IT  & $>$        & $>$        & 5030.0   & 24.4    & 178.8    & 89.7     \\
     &      &            &      &      &      &            &      & CPU & $>$       & $>$       & 53.9234  & 0.2553  & 0.2133   & 0.1048   \\
\midrule
1000 & 200  & 100        & 200  & 200  & 1000 & 100        & 50   & IT  & $>$        &$>$        & -        & 24.3    & 206.3    & 96.3     \\
     &      &            &      &      &      &            &      & CPU & $>$       &$>$       & -        & 0.1701  & 0.1779   & 0.0753   \\
\midrule
2000 & 400  & 200        & 400  & 400  & 2000 & 400        & 200  & IT  & $>$       &$>$       &$>$         & 23.7    & 195.5    & 95.1     \\
     &      &            &      &      &      &            &      & CPU & $>$       & $>$       & $>$        & 0.9242  & 1.4372   & 0.6825   \\
\midrule
2000 & 400  & 200        & 400  & 400  & 2000 & 200        & 100  & IT  & $>$       & $>$       & -        & 23.7    & 216.2    & 98.2     \\
     &      &            &      &      &      &            &      & CPU & $>$       & $>$       & -        & 0.7470  & 1.2880   & 0.5734   \\
\midrule
5000 & 1000 & 750        & 500  & 1000 & 5000 & 1000       & 500  & IT  & $>$       &$>$       &$>$        & 52.6    & 318.7    & 163.6    \\
     &      &            &      &      &      &            &      & CPU & $>$      & $>$       &$>$        & 13.2482 & 20.7552  & 10.2175  \\
\midrule
5000 & 1000 & 750        & 500  & 1000 & 5000 & 750        & 500  & IT  & $>$       & $>$       & -        & 35.3    & 279.2    & 134.5    \\
     &      &            &      &      &      &            &      & CPU & $>$       & $>$       & -        & 8.6873  & 17.3383  & 8.1562\\
\bottomrule
\end{tabular}}
\label{tab1}
\end{table}

\begin{table}[!htb]
\centering
	\caption{The average IT and CPU of RK, GRK, RBCD, GRBK, GRABK-c, and GRABK-a for matrix equations with Type II.}
	\scalebox{0.7}{
\begin{tabular}{lllllllllllll}
\toprule
$m$    & $p$   & $\tau_1$ & $q$    & $n$    & $\tau_2$ &     & RK      & GRK     & RBCD     & GRBK    & GRABK-c & GRABK-a \\
\midrule
50   & 20   & 10  & 20   & 50   & 10  & IT  & 45699.0 & 45196.0 & 3909.1   & 177.3   & 1574.0  & 700.7   \\
     &      &     &      &      &     & CPU & 1.6867  & 1.1436  & 0.1296   & 0.0237  & 0.0532  & 0.0252  \\
\midrule
100  & 40   & 20  & 40   & 100  & 20  & IT  & $>$        & $>$        & 9843.7   & 248.0   & 2855.7  & 1075.9  \\
     &      &     &      &      &     & CPU & $>$        & $>$        & 0.7973   & 0.1001  & 0.1465  & 0.0547  \\
\midrule
500  & 100  & 100 & 40   & 100  & 20  & IT  & $>$        &$>$        & 3911.6   & 43.8    & 869.6   & 358.2   \\
     &      &     &      &      &     & CPU &$>$        & $>$        & 1.1979   & 0.0873  & 0.1466  & 0.0491  \\
\midrule
500  & 100  & 100 & 100  & 500  & 50  & IT  & $>$        & $>$        & 4525.0     & 26.8    & 383.9   & 175.0     \\
     &      &     &      &      &     & CPU & $>$        & $>$        & 8.4209   & 0.0790  & 0.1067  & 0.0497  \\
\midrule
1000 & 200  & 200 & 100  & 500  & 50  & IT  & $>$        & $>$        & 4988.4   & 28.7    & 429.1   & 188.4   \\
     &      &     &      &      &     & CPU &$>$        &$>$        & 30.0199  & 0.2781  & 0.2926  & 0.1276  \\
\midrule
1000 & 200  & 200 & 200  & 1000 & 100 & IT  & $>$        & $>$        & 9893.0   & 27.3    & 415.1   & 182.2   \\
     &      &     &      &      &     & CPU & $>$        & $>$        & 110.7012 & 0.3463  & 0.5131  & 0.2240  \\
\midrule
5000 & 1000 & 500 & 1000 & 5000 & 500 & IT  & $>$        &$>$     & $>$         & 99.2    & 614.5   & 285.8   \\
     &      &     &      &      &     & CPU & $>$      & $>$        &$>$         & 25.8886 & 40.5165 & 18.2868\\
\midrule
\end{tabular}}
\label{tab2}
\end{table}

In \Cref{tab1,tab2}, we report the average IT and CPU of RK, GRK, RBCD, GRBK, GRABK-c, and GRABK-a for solving matrix equations with Type I and Type II matrices. For the RBCD \cite{Du22}, we use the same parameters as in reference \cite{Du22}. For the GRBK, GRABK\_c, and GRABK\_a, we use the same block partition and block size. For the GRABK-c, and GRABK-a, the stepsizes $\alpha=\tfrac{1.95}{\beta_{\max}^2(A)\beta_{\max}^2(B)}$ and $\alpha_k=L_k$ are used, respectively. In the following tables, the item `$>$' represents that the number of iteration steps exceeds the maximum iteration (50000 or 120s). The item `-' represents that the method does not converge. From these two tables, we can see the following phenomena.

The GRK method is better than the RK method in terms of  IT and CPU time. The IT and CPU time of both the RK and GRK methods increases with the increase of matrix dimension. However, the GRK method has a small increase in terms of CPU time.

The GRBK, GRABK-c, and GRABK-a methods vastly outperform the RBCD method in terms of  IT and CPU time. The RBCD method does not converge in the case that the matrix $B$ is not full row rank. When the matrix size is small, the GRBK method is competitive, because the calculation of the pseudoinverse is less expensive and the number of iteration steps is small. When the matrix size is large, the GRABK-a method is more challenging, because the method does not need to calculate the pseudoinverse, and the step size is adaptive.

In \Cref{fig3}, we plot the relative error of GRABK-c with a fixed stepsize $\alpha=\tfrac{1.95}{\beta_{\max}^2(A)\beta_{\max}^2(B)}$
and different block sizes $\tau_1=\tau_2=\tau$ for two matrix equation with Type I ($A=U_1D_1V_1^T$ with $m=500, p=250,r_1=150$ and $B=U_2D_2V_2^T$ with $n=500, q=250, r_2=150$) and Type II ($A={\rm randn}(500,250)$ and $B={\rm randn}(250,500)$).   Similarly, in \Cref{fig4}, we plot the relative error of GRABK-a with a fixed fixed stepsize  $\alpha_k=L_k$ and different block size $\tau_1=\tau_2=\tau$ for two matrix equation with Type I ($A=U_1D_1V_1^T$ with $m=500, p=250,r_1=150$ and $B=U_2D_2V_2^T$ with $n=500, q=250, r_2=150$) and Type II ($A={\rm randn}(500,250)$ and $B={\rm randn}(250,500)$). From \Cref{fig3,fig4}, we observed that increasing block size leads to a better convergence rate of the GRABK-c and GRABK-a methods. As the block size $\tau$ increases, the IT and CPU time first decreases, and then increases after reaching the minimum. This means that a proper block size $\tau$ can speed up the convergence speed of the GRABK-c and GRABK-a methods. If the GRABK-c and GRABK-a methods are implemented in parallel, the larger the block size $\tau$ will be the better.

\begin{figure}[!htb]
	\centering
	\subfigure[Type I]{\includegraphics[width=5in]{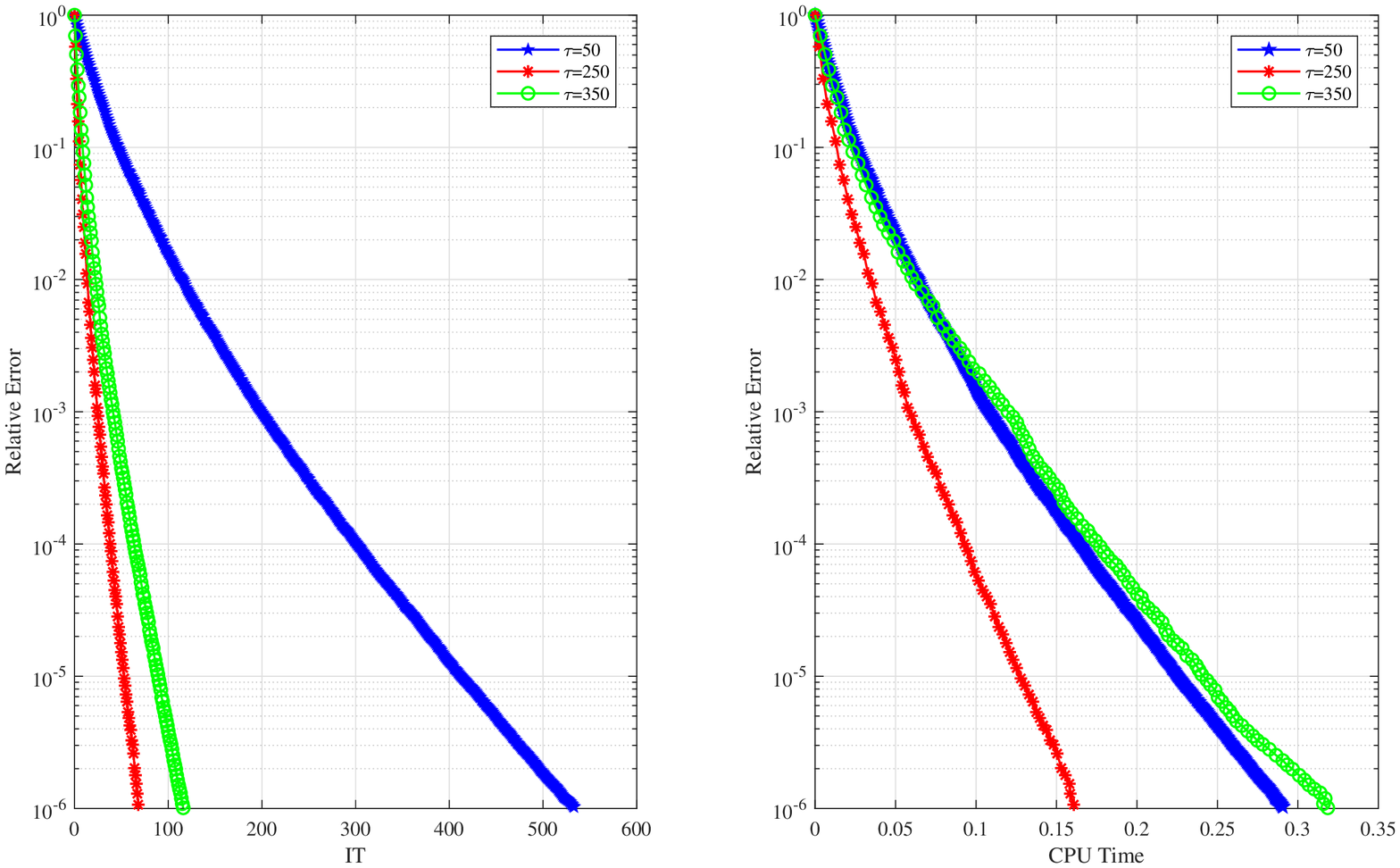}}
	\subfigure[Type II]{\includegraphics[width=5in]{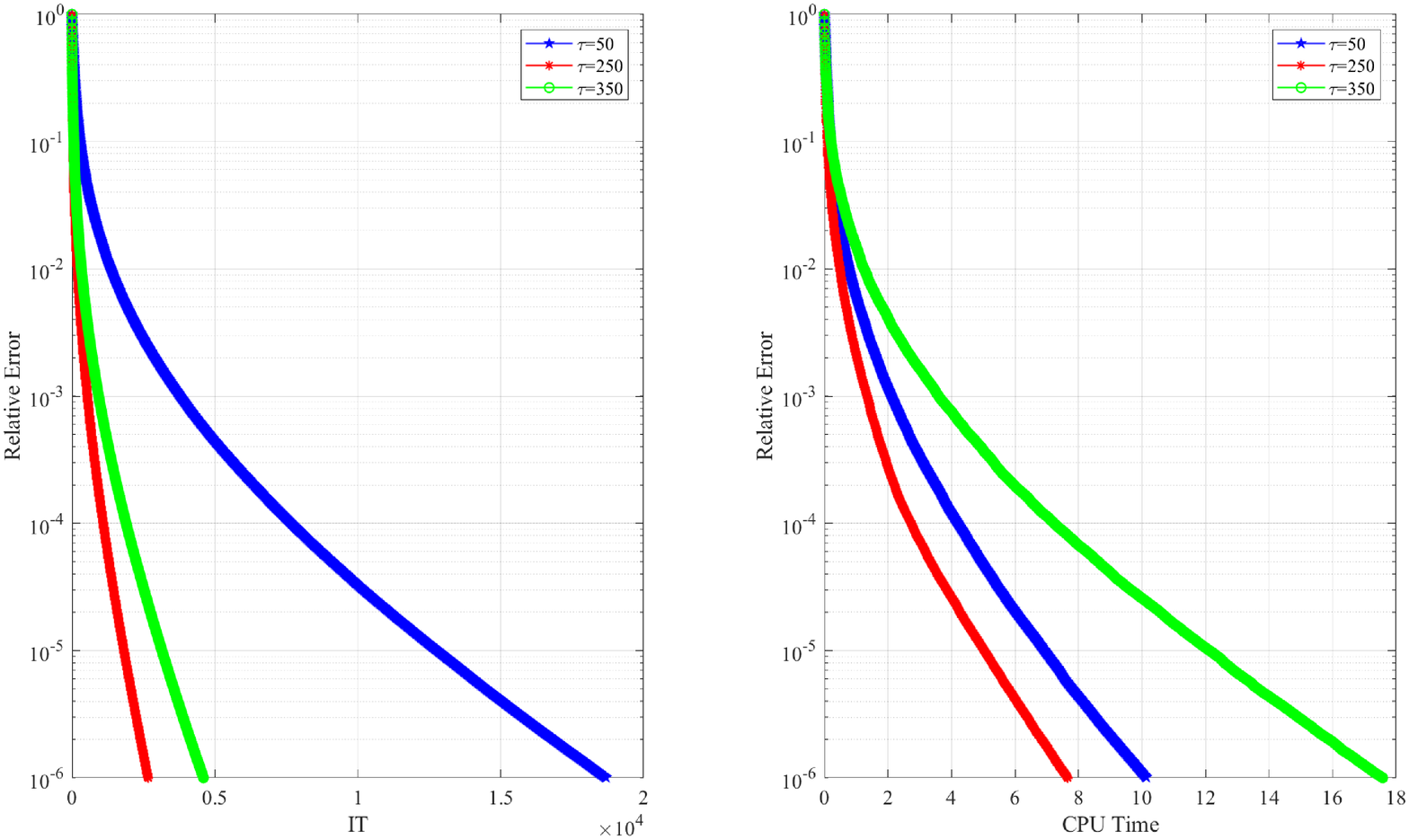}}
	\caption{The relative error of GRABK-c with stepsize $\alpha=\tfrac{1.95}{\beta_{\max}^2(A)\beta_{\max}^2(B)}$ and different block size $\tau_1=\tau_2=\tau$ for two matrix equation.}
	\label{fig3}
\end{figure}

\begin{figure}[!htb]
	\centering
	\subfigure[Type I]{\includegraphics[width=5in]{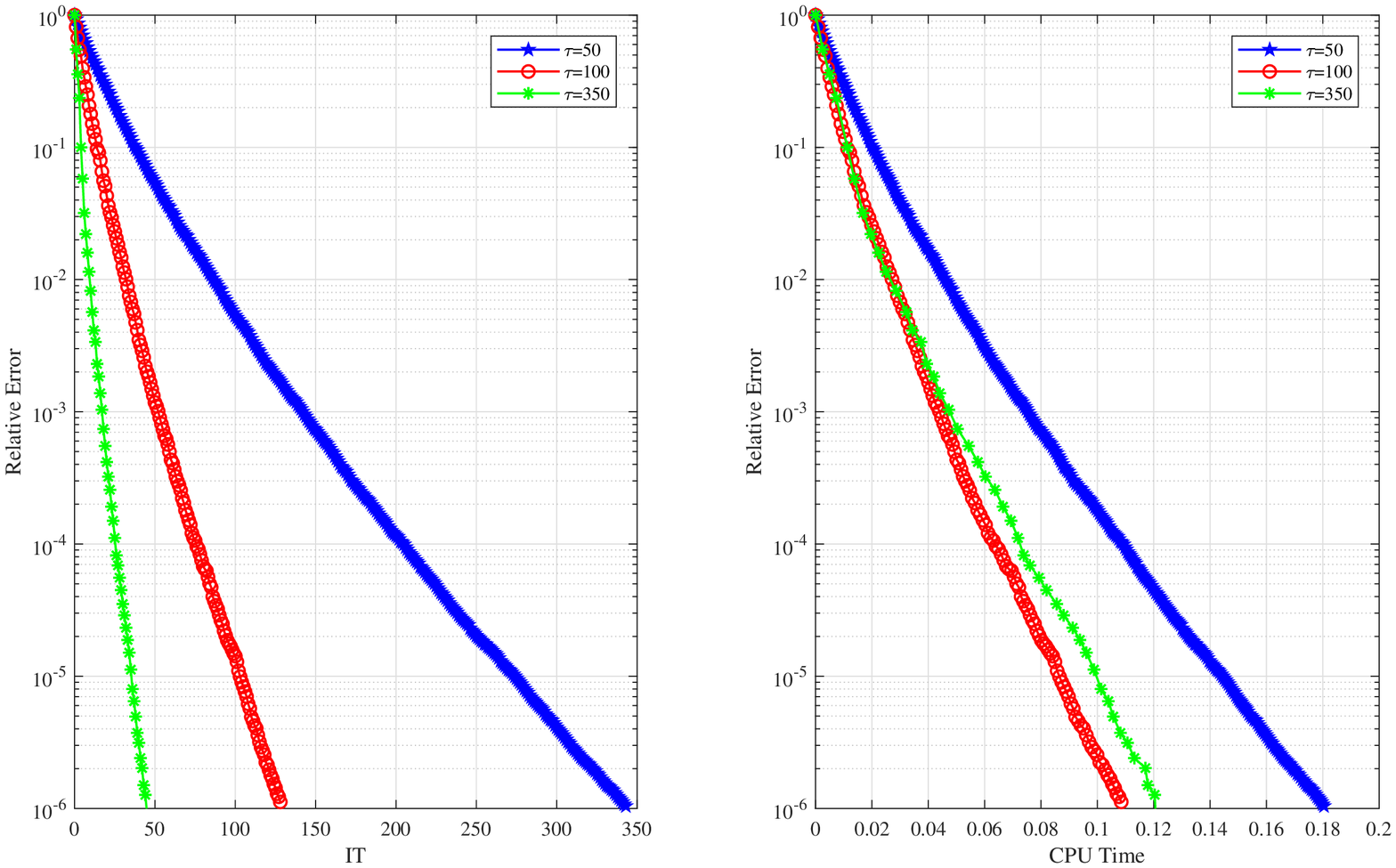}}
	\subfigure[Type II]{\includegraphics[width=5in]{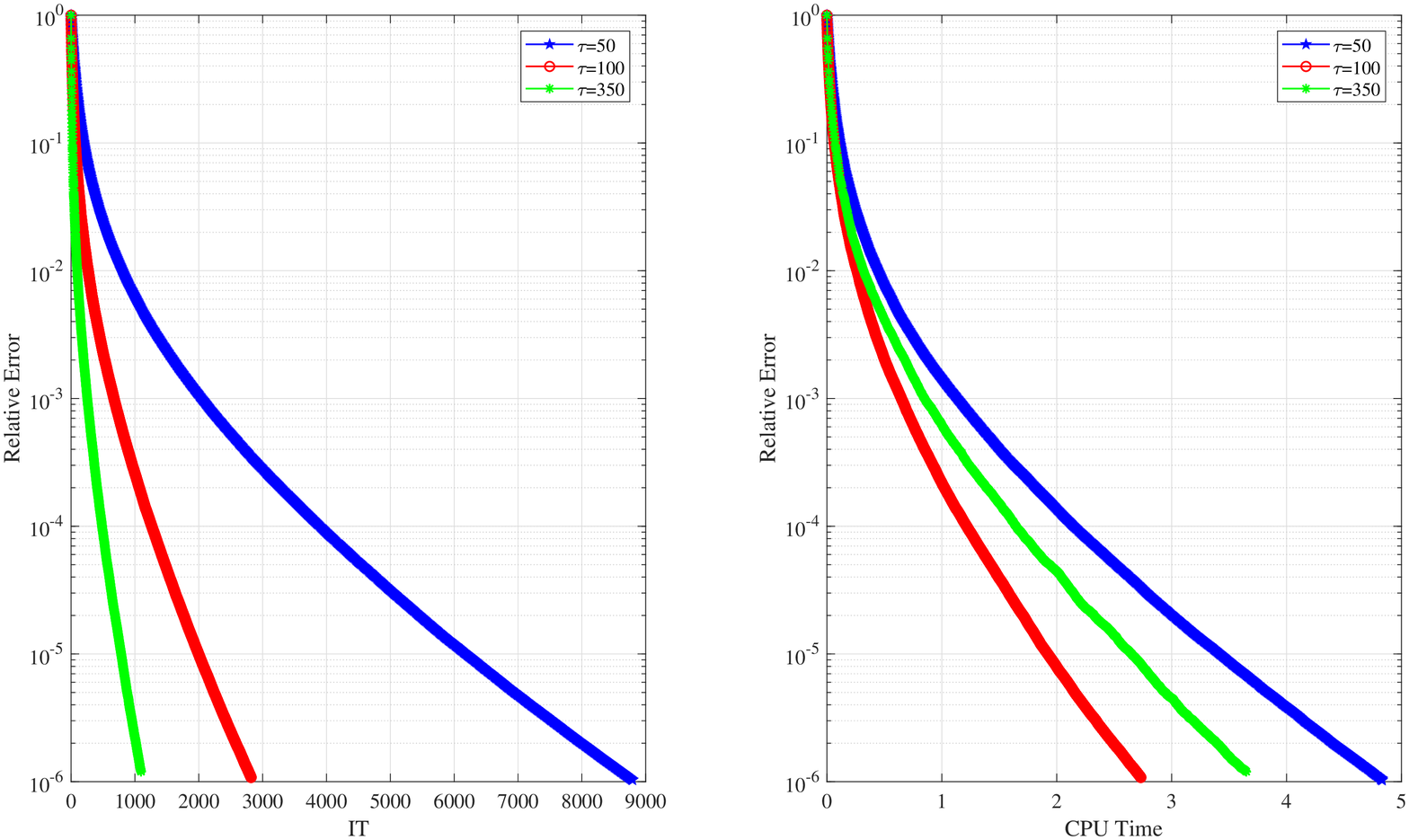}}
	\caption{The relative error of GRABK-a with stepsize $\alpha_k=L_k$ and different block size $\tau_1=\tau_2=\tau$  for two matrix equation.}
	\label{fig4}
\end{figure}

\subsection{Real-world sparse data}
We also test the RK, GRK, RBCD, GRBK, GRABK-c, and GRABK-a methods for solving matrix equations with sparse matrices from the Florida sparse matrix collection \cite{Davis2011}. In \Cref{tab3}, we report the average IT and CPU of RK, GRK, RBCD, GRBK, GRABK-c, and GRABK-a for solving matrix equations with sparse matrices. The parameters of these methods are the same as in the previous subsection. We observe that the GRK method out requires less IT and CPU than the RK method. The GRBK, GRABK-c, and GRABK-a methods vastly outperform the RBCD method in terms of IT and CPU time. Hence, the GRBK, GRABK-c, and GRABK-a methods are competitive, because the product of a sparse matrix and a vector is less expensive. When the matrix size is large, the GRABK-c and GRABK-a methods can be implemented in parallel.

\begin{table}[!htbp]
\centering
	\caption{The average IT and CPU of RK, GRK, RBCD, GRBK, GRABK-c, and GRABK-a for matrix equations with sparse matrices from \cite{Davis2011}.}
	\scalebox{0.6}{
\begin{tabular}{lllllllllllllll}
\toprule
$m$    & $p$    & $r_1$& $\tau_1$ & $q$    & $n$    & $r_2$ & $\tau_2$ &     & RK      & GRK     & RBCD     & GRBK    & GRABK-c & GRABK-a \\
\midrule
\multicolumn{4}{l}{ash219}  & \multicolumn{4}{l}{n4c5-b11}    & IT  & 19446.6        & 18920.1        & 317.6    & 663.7    & 3007.1   & 1237.9   \\
219    & 85   & 85   & 20   & 10     & 120    & 10     & 5    & CPU & 4.1239         & 0.5464         & 0.0562   & 0.0883   & 0.1173   & 0.0505   \\
\multicolumn{4}{l}{ash219}  & \multicolumn{4}{l}{relat4$^T$}     & IT  & 37283.6        & 38512.3        & -        & 408.6    & 4272.3   & 1519.3   \\
219    & 85   & 85   & 20   & 12     & 66     & 5      & 5    & CPU & 5.5010         & 1.0767         & -        & 0.1077   & 0.1774   & 0.0601   \\
\multicolumn{4}{l}{rel4}    & \multicolumn{4}{l}{n4c5-b11}    & IT  & 2457.1         & 2382.9         & 679.4    & 187.3    & 968.0    & 20355.6  \\
66     & 12   & 5    & 5    & 10     & 120    & 10     & 5    & CPU & 0.2004         & 0.0618         & 0.0264   & 0.0178   & 0.0309   & 1.7052   \\
\multicolumn{4}{l}{rel4}    & \multicolumn{4}{l}{relat4$^T$}     & IT  & 5399.9         & 5409.5         & -        & 288.8    & 2801.7   & 688.7    \\
66     & 12   & 5    & 5    & 12     & 66     & 5      & 5    & CPU & 0.3149         & 0.1268         & -        & 0.0241   & 0.0826   & 0.0210   \\
\multicolumn{4}{l}{mk10-b1} & \multicolumn{4}{l}{bibd\_11\_5} & IT  & $>$ & $>$ & 1214     & 1123.8   & 3242.8   & 1862.3   \\
630    & 45   & 44   & 50   & 55     & 462    & 55     & 50   & CPU & $>$ & $>$ & 1.7975 & 0.7687 & 0.3273 & 0.1882 \\
\multicolumn{4}{l}{mk10-b1} & \multicolumn{4}{l}{rel5$^T$}       & IT  & $>$ &$>$ & -        & 595.4    & 10102.5  & 1971.6   \\
630    & 45   & 44   & 50   & 35     & 340    & 24     & 30   & CPU & $>$&$>$ & -        & 0.2681   & 0.8224   & 0.1625   \\
\multicolumn{4}{l}{relat5}  & \multicolumn{4}{l}{bibd\_11\_5} & IT  & $>$ & $>$ & 11985    & 812      & 44387    & 4378     \\
340    & 35   & 24   & 30   & 55     & 462    & 55     & 50   & CPU & $>$ &$>$& 10.5393  & 0.4626   & 3.6697   & 0.4041   \\
\multicolumn{4}{l}{relat5}  & \multicolumn{4}{l}{rel5$^T$}       & IT  & $>$&$>$ & -        & 554.3    & 91091.1  & 4597.0   \\
340    & 35   & 24   & 30   & 35     & 340    & 24     & 30   & CPU & $>$ & $>$ & -        & 0.1967   & 6.3276   & 0.3362   \\
\multicolumn{4}{l}{relat6}  & \multicolumn{4}{l}{ash958$^T$}     & IT  &$>$ &$>$ & 122525.3 &490.5 &90579.2&7304.7  \\
2340   & 157  & 137  & 200  & 292    & 958    & 292    & 200  & CPU & $>$&$>$ &218.4782&	4.7630& 	35.6329& 	13.7720  \\
\bottomrule
\end{tabular}}
\label{tab3}
\end{table}

As we can see from the numerical results, the GRBK, GRABK-c, and GRABK-a methods require less IT than the RBCD method for dense and sparse matrices. The GRBK, GRABK-c, and GRABK-a methods require fewer CPU times than the RBCD method. The RBCD method does not converge in the case that the matrix $B$ is not full row rank. The GRBK, GRABK-c, and GRABK-a methods can be implemented more efficiently than the RBCD  method in many computer architectures, and the GRABK-c and GRABK-a methods can be deployed on parallel computing units to reduce the computational time.

\subsection{An application to image restoration}
In this section, we illustrate the effectiveness of our proposed method with several examples of image restoration \cite{Hansen}. Let $X^*=(x_{ij}^*)_{p\times q}$ and $X=(x_{ij})_{p\times q}$ be the original and restored images, respectively. The quality of the restoration result is compared by using the peak signal-to-noise ratio (PSNR) with the following formula
\begin{equation*}
{\rm PSNR}=10\log_{10}\left(\tfrac{(\max\{x_{ij}^*\})^2}{\rm MSE}\right),
\end{equation*}
where ${\rm MSE}=\tfrac{\sum_{i=1}^{p}\sum_{j=1}^{q}(x^*_{ij}-x_{ij})^2}{pq}$. The original image is represented by an array $n\times n$ pixels. We consider the matrix equation \eqref{eq:systems}, where $C$ is the observed blurred image, and the blurring matrices $A$ is the uniform Toeplitz matrices of size $n\times n$ defined by
\begin{equation*}
a_{ij}=\begin{cases}
           \tfrac{1}{2r-1}, & |i-j|\leq r \\
           0, & \mbox{otherwise},
         \end{cases}
\end{equation*}
and $B$ is the Gaussian Toeplitz matrices of size $n\times n$ given by
\begin{equation*}
b_{ij}=\begin{cases}
           \tfrac{1}{\sigma\sqrt{2\pi}}\exp\left(-\tfrac{(i-j)^2}{2\sigma^2}\right), & |i-j|\leq r \\
           0, & \mbox{otherwise}.
         \end{cases}
\end{equation*}
For $r=3$ and $\sigma=7$, we apply RBCD, GRBK, GRABK-c with stepsize $\alpha=\tfrac{1.95}{\beta_{\max}^2(A)\beta_{\max}^2(B)}$, and GRABK-a with stepsize $\alpha_k=L_k$ for this example, and  block size $\tau_1=\tau_2=\tfrac{n}{2}$. In \Cref{fig5,fig6,fig7,fig8}, the original, blurred, and de-blurred images are shown.

\begin{figure}[!htb]
	\centering
	\subfigure{\includegraphics[width=5in]{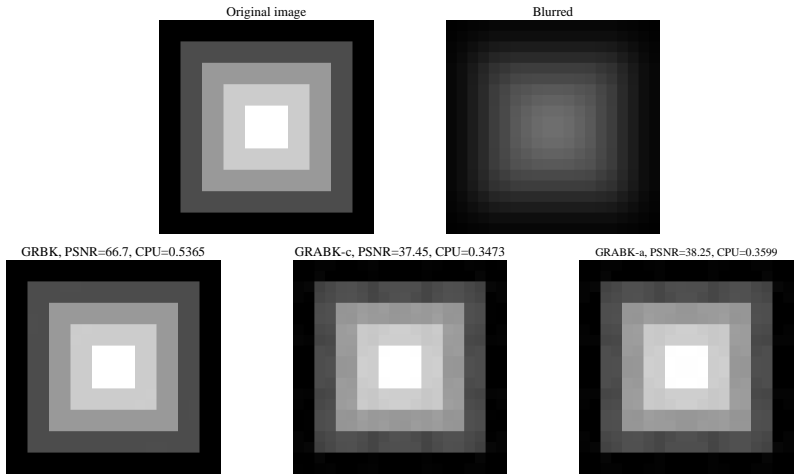}}
	\caption{Digital image, $n=20$.}
	\label{fig5}
\end{figure}

\begin{figure}[!htb]
	\centering
	\subfigure{\includegraphics[width=5in]{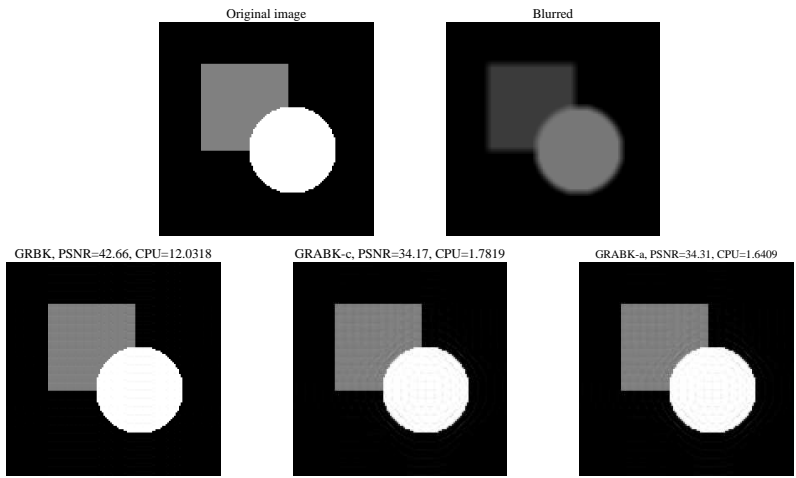}}
	\caption{Digital image, $n=128$.}
	\label{fig6}
\end{figure}

\begin{figure}[!htb]
	\centering
	\subfigure{\includegraphics[width=5in]{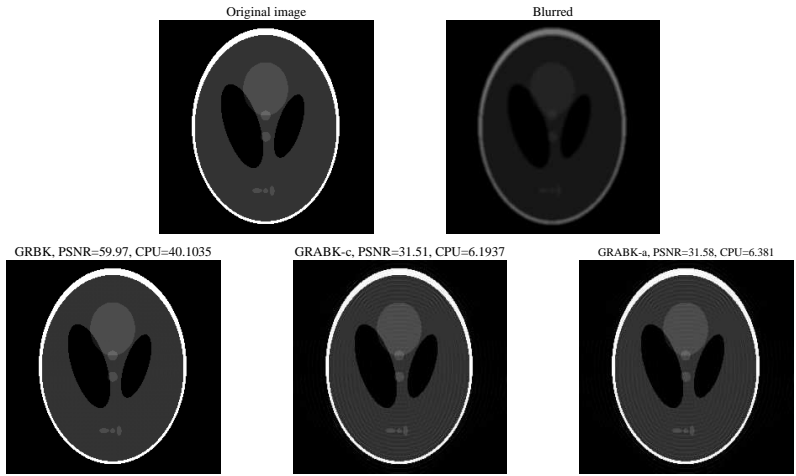}}
	\caption{Modified Shepp-Logan, $n=256$.}
	\label{fig7}
\end{figure}

\begin{figure}[!htb]
	\centering
	\subfigure{\includegraphics[width=5in]{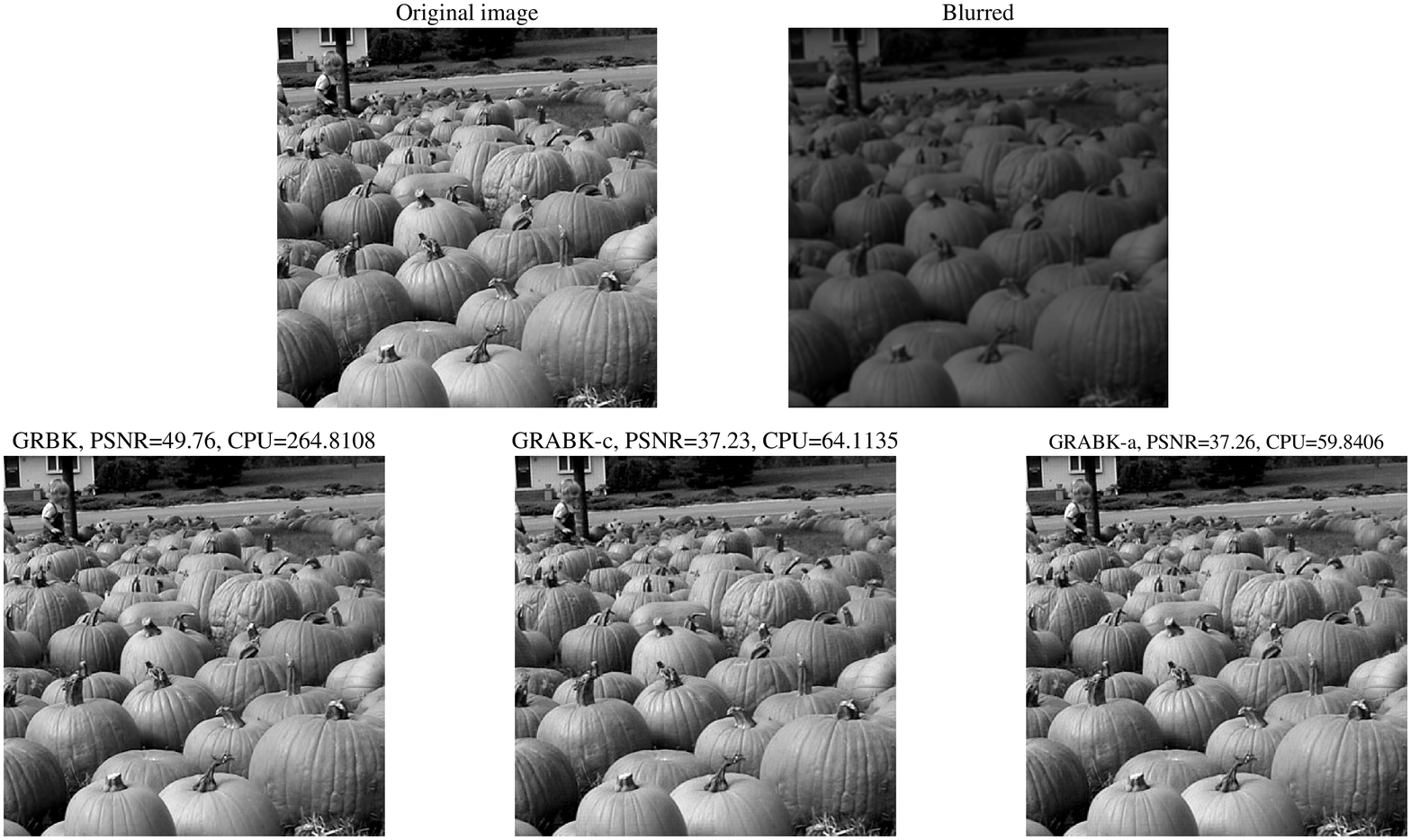}}
	\caption{Pumpkins, $n=512$.}
	\label{fig8}
\end{figure}

The blurring of these images is caused by the out-of-focus and atmospheric turbulence. The capability of our proposed methods is confirmed by de-blurring the blurred image. From \Cref{fig5,fig6,fig7,fig8}, we see that the PSNR of the restored image by the GRBK method is better than those by GRABK-c, and GRABK-a. However, the GRBK method requires more CPU times than the GRABK-c, and GRABK-a, because the GRBK method needs to calculate the pseudoinverse. When the image is large, the GRABK-c and GRABK-a methods can restore the image better and require fewer CPU times, and can be implemented in parallel.

\section{Conclusions}\label{Sec:Con}
We have proposed the global randomized block Kaczmarz method to solve large-scale matrix equations and prove its convergence theory. In addition, we also have presented a class of global randomized average block Kaczmarz methods for solving large-scale matrix equations, which provide a general framework for the design and analysis of global randomized block Kaczmarz method to solve large-scale matrix equations. Our convergence results provide a theoretical guarantee for the convergence of the global randomized average block Kaczmarz methods with constant and adaptive stepsizes. The numerical examples also illustrate the benefits of the new methods.

%\bibliographystyle{siamplain}
%\bibliography{Matrix_equation}

\begin{thebibliography}{10}

\bibitem{Cvetk2008}
{\sc D.~S. Cvetkovic-Ilic}, {\em Re-nnd solutions of the matrix equation
  {$AXB=C$}}, J. Aust. Math. Soc., 84 (2008), pp.~63--72.

\bibitem{Davis2011}
{\sc T.~A. Davis and Y.~Hu}, {\em The university of {F}lorida sparse matrix
  collection}, ACM Trans. Math. Softw., 38 (2011), pp.~1--25.

\bibitem{Ding2005}
{\sc F.~Ding and T.~Chen}, {\em Iterative least-squares solutions of coupled
  {S}ylvester matrix equations}, Syst. Control Lett., 54 (2005), pp.~95--107.

\bibitem{Ding2008}
{\sc F.~Ding, P.~X. Liu, and J.~Ding}, {\em Iterative solutions of the
  generalized {S}ylvester matrix equations by using the hierarchical
  identification principle}, Appl. Math. Comput., 197 (2008), pp.~41--50.

\bibitem{Du2019}
{\sc K.~Du}, {\em Tight upper bounds for the convergence of the randomized
  extended {K}aczmarz and {G}auss-{S}eidel algorithms}, Numer. Linear Algebra
  Appl., 26 (2019), p.~e2233.

\bibitem{Du22}
{\sc K.~Du, C.-C. Ruan, and X.-H. Sun}, {\em On the convergence of a randomized
  block coordinate descent algorithm for a matrix least squares problem}, Appl.
  Math. Lett., 124 (2022), p.~107689.

\bibitem{Du20}
{\sc K.~Du, W.-T. Si, and X.-H. Sun}, {\em Randomized extended average block
  {K}aczmarz for solving least squares}, SIAM J. Sci. Comput., 42 (2020),
  pp.~A3541--A3559.

\bibitem{Du21}
{\sc K.~Du and X.-H. Sun}, {\em A doubly stochastic block {G}auss-{S}eidel
  algorithm for solving linear equations}, Appl. Math. Comput., 408 (2021),
  p.~126373.

\bibitem{Fausett1994}
{\sc D.~W. Fausett and C.~T. Fulton}, {\em Large least squares problems
  involving {K}ronecker products}, SIAM J. Matrix Anal. Appl., 15 (1994),
  pp.~219--227.

\bibitem{Gower}
{\sc R.~M. Gower and P.~Richt\'{a}rik}, {\em Randomized iterative methods for
  linear systems}, SIAM. J. Matrix Anal. Appl., 36 (2015), pp.~1660--1690.

\bibitem{Graham1981}
{\sc A.~Graham}, {\em Kronecker Products and Matrix Calculus: With
  Applications}, Wiley, New York, 1981.

\bibitem{Hansen}
{\sc P.~C. Hansen, J.~G. Nagy, and D.~P. O'Leary}, {\em Deblurring Images:
  Matrices, Spectra and Filtering}, SIAM, Philadelphia, 2006.

\bibitem{Leventhal2008}
{\sc D.~Leventhal and A.~S. Lewis}, {\em Randomized methods for linear
  constraints: Convergence rates and conditioning}, Math. Oper. Res., 35
  (2008), pp.~641--654.

\bibitem{Ma2015}
{\sc A.~Ma, D.~Needell, and A.~Ramdas}, {\em Convergence properties of the
  randomized extended {G}auss-{S}eidel and {K}aczmarz methods}, SIAM J. Matrix
  Anal. Appl., 36 (2015), pp.~1590--1604.

\bibitem{Needell2015}
{\sc D.~Needell, R.~Zhao, and A.~Zouzias}, {\em Randomized block {K}aczmarz
  method with projection for solving least squares}, Linear Algebra Appl., 484
  (2015), pp.~322 -- 343.

\bibitem{Peng2010-2}
{\sc Z.-Y. Peng}, {\em A matrix {LSQR} iterative method to solve matrix
  equation {$AXB=C$}}, Int. J. Comput. Math., 87 (2010), pp.~1820--1830.

\bibitem{Rauhala1980}
{\sc U.~Rauhala}, {\em Introduction to array algebra}, Photogramm. Eng. Remote
  Sens., 46 (1980), pp.~177--192.

\bibitem{Regalia1989}
{\sc P.~A. Regalia and S.~K. Mitra}, {\em Kronecker products, unitary matrices
  and signal processing applications}, SIAM Rew., 31 (1989), pp.~586--613.

\bibitem{Strohmer2009}
{\sc T.~Strohmer and R.~Vershynin}, {\em A randomized {K}aczmarz algorithm with
  exponential convergence}, J. Fourier Anal. Appl., 15 (2009), pp.~262--278.

\bibitem{Tian2017}
{\sc Z.~Tian, M.~Tian, Z.~Liu, and T.~Xu}, {\em The {J}acobi and
  {G}auss-{S}eidel-type iteration methods for the matrix equation {$AXB = C$}},
  Appl. Math. Comput., 292 (2017), pp.~63--75.

\bibitem{Chu1987}
{\sc K.~wah Eric~Chu}, {\em Singular value and generalized singular value
  decompositions and the solution of linear matrix equations}, Linear Algebra
  Appl., 88-89 (1987), pp.~83--98.

\bibitem{Wei2016}
{\sc Y.~Wei, P.~Xie, and L.~Zhang}, {\em Tikhonov regularization and randomized
  {GSVD}}, SIAM J. Matrix Anal. Appl., 37 (2016), pp.~649--675.

\bibitem{Wei2019}
{\sc P.~Xie, H.~Xiang, and Y.~Wei}, {\em Randomized algorithms for total least
  squares problems}, Numer. Linear Algebra Appl., 26 (2019), p.~e2219.

\bibitem{Peng2005}
{\sc Z.~yun Peng}, {\em An iterative method for the least squares symmetric
  solution of the linear matrix equation {$AXB=C$}}, Appl. Math. Comput., 170
  (2005), pp.~711--723.

\bibitem{Peng2010-1}
{\sc Z.~yun Peng}, {\em New matrix iterative methods for constraint solutions
  of the matrix equation {$AXB=C$}}, J. Comput. Appl. Math., 235 (2010),
  pp.~726--735.

\bibitem{Zha1995}
{\sc H.~Y. Zha}, {\em Comments on large least squares problems involving
  {K}ronecker products}, SIAM J. Matrix Anal. Appl.,  (1995), p.~1172.

\bibitem{Zhang2011}
{\sc F.~Zhang, Y.~Li, W.~Guo, and J.~Zhao}, {\em Least squares solutions with
  special structure to the linear matrix equation {$AXB=C$}}, Appl. Math.
  Comput., 217 (2011), pp.~10049--10057.

\bibitem{Wei2020}
{\sc L.~Zhang and Y.~Wei}, {\em Randomized core reduction for discrete
  ill-posed problem}, J. Comput. Appl. Math., 375 (2020), p.~112797.

\bibitem{Zouzias2012}
{\sc A.~Zouzias and N.~Freris}, {\em Randomized extended {K}aczmarz for solving
  least squares}, SIAM J. Matrix Anal. Appl., 34 (2012), pp.~773--793.

\end{thebibliography}

\end{document}